\documentclass[a4paper, 12pt]{amsart}

\usepackage{mathptmx, amsmath, tabularx, amsthm, amssymb, amsfonts, amscd, mathrsfs}
\usepackage{eulervm}
\usepackage{paralist}
\usepackage{aliascnt}
\usepackage{blkarray}
\usepackage{mathbbol}
\usepackage[inner=2.5cm,outer=2.5cm, bottom=3.3cm]{geometry}
\usepackage{setspace}
\usepackage[colorlinks=true,linkcolor=blue,urlcolor=blue,pagebackref]{hyperref}
\usepackage[initials]{amsrefs}
\usepackage{tikz}
\usetikzlibrary{matrix}
\usetikzlibrary{arrows,calc}

\BibSpec{collection.article}{%
	+{}  {\PrintAuthors}                {author}
	+{,} { \textit}                     {title}
	+{.} { }                            {part}
	+{:} { \textit}                     {subtitle}
	+{,} { \PrintContributions}         {contribution}
	+{,} { \PrintConference}            {conference}
	+{}  {\PrintBook}                   {book}
	+{,} { }                            {booktitle}
	+{,} { }                            {series}
	+{, vol.} { }                            {volume}
	+{,} { }                            {publisher}
	+{,} { \PrintDateB}                 {date}
	+{,} { pp.~}                        {pages}
	+{,} { }                            {status}
	+{,} { \PrintDOI}                   {doi}
	+{,} { available at \eprint}        {eprint}
	+{}  { \parenthesize}               {language}
	+{}  { \PrintTranslation}           {translation}
	+{;} { \PrintReprint}               {reprint}
	+{.} { }                            {note}
	+{.} {}                             {transition}
	+{}  {\SentenceSpace \PrintReviews} {review}
}

\newcommand{\kk}{\mathbb{k}}

\newcommand{\CC}{\mathbb{C}}

\newcommand{\NN}{\normalfont\mathbb{N}}
\newcommand{\ZZ}{\normalfont\mathbb{Z}}

\newcommand{\PP}{{\normalfont\mathbb{P}}}

\newcommand{\ara}{{\normalfont\text{ara}}}

\newcommand{\mm}{{\normalfont\mathfrak{m}}}

\newcommand{\QQ}{\mathbb{Q}}
\newcommand{\pp}{{\normalfont\mathfrak{p}}}
\newcommand{\aaa}{\mathfrak{a}}
\newcommand{\bbb}{\mathfrak{b}}

\newcommand{\nn}{{\normalfont\mathfrak{N}}}
\newcommand{\bn}{{\normalfont\mathbf{n}}}

\newcommand{\bx}{{\normalfont\mathbf{x}}}

\newcommand{\bm}{{\normalfont\mathbf{m}}}
\newcommand{\ttt}{{\normalfont\mathbf{t}}}

\newcommand{\fP}{{\mathfrak{P}}}
\newcommand{\fJ}{{\mathfrak{J}}}
\newcommand{\fL}{{\mathfrak{L}}}

\newcommand{\fK}{{\mathfrak{K}}}
\newcommand{\fS}{{\mathfrak{S}}}
\newcommand{\fT}{{\mathfrak{T}}}
\newcommand{\rank}{\normalfont\text{rank}}

\newcommand{\grade}{\normalfont\text{grade}}

\newcommand{\codim}{\normalfont\text{codim}}

\newcommand{\sdim}{{\normalfont\text{sdim}}}

\newcommand{\Ker}{\normalfont\text{Ker}}

\newcommand{\Quot}{\normalfont\text{Quot}}

\newcommand{\HT}{\normalfont\text{ht}}

\newcommand{\Ann}{\normalfont\text{Ann}}
\newcommand{\Supp}{\normalfont\text{Supp}}

\newcommand{\Min}{{\normalfont\text{Min}}}
\newcommand{\Sym}{\normalfont\text{Sym}}
\newcommand{\Rees}{\mathcal{R}}

\newcommand{\ee}{{\normalfont\mathbf{e}}}

\newcommand{\OO}{\mathcal{O}}
\newcommand{\mO}{\mathscr{O}}

\newcommand{\LL}{\mathbb{L}}

\newcommand{\FF}{\normalfont\mathcal{F}}

\newcommand{\gr}{{\normalfont\text{gr}}}

\newcommand{\Proj}{\normalfont\text{Proj}}
\newcommand{\Hilb}{{\normalfont\text{Hilb}}}
\newcommand{\Spec}{\normalfont\text{Spec}}

\newcommand{\multProj}{\normalfont\text{MultiProj}}
\newcommand{\msupp}{\normalfont\text{MSupp}}

\def\f0{\mathbf{0}}
\def\fone{\mathbf{1}}

\def\fs{\mathbf{s}}

\def\fn{\mathbf{n}}

\def\ls{\leqslant}
\def\gs{\geqslant}

\newtheorem{theorem}{Theorem}[section]

\newtheorem{headthm}{Theorem}

\newaliascnt{headcor}{headthm}

\aliascntresetthe{headcor}

\newaliascnt{headthmdef}{headthm}

\aliascntresetthe{headthmdef}

\newaliascnt{headconj}{headthm}

\aliascntresetthe{headconj}

\newaliascnt{corollary}{theorem}
\newtheorem{corollary}[corollary]{Corollary}
\aliascntresetthe{corollary}

\newaliascnt{lemma}{theorem}
\newtheorem{lemma}[lemma]{Lemma}
\aliascntresetthe{lemma}

\newaliascnt{conjecture}{theorem}

\aliascntresetthe{conjecture}

\newaliascnt{proposition}{theorem}
\newtheorem{proposition}[proposition]{Proposition}
\aliascntresetthe{proposition}

\theoremstyle{definition}
\newaliascnt{definition}{theorem}
\newtheorem{definition}[definition]{Definition}
\aliascntresetthe{definition}

\newaliascnt{notation}{theorem}
\newtheorem{notation}[notation]{Notation}
\aliascntresetthe{notation}

\newaliascnt{example}{theorem}
\newtheorem{example}[example]{Example}
\aliascntresetthe{example}

\newaliascnt{examples}{theorem}

\aliascntresetthe{examples}

\newaliascnt{remark}{theorem}
\newtheorem{remark}[remark]{Remark}
\aliascntresetthe{remark}

\newaliascnt{problem}{theorem}

\aliascntresetthe{problem}

\newaliascnt{question}{theorem}
\newtheorem{question}[question]{Question}
\aliascntresetthe{question}

\newaliascnt{convention}{theorem}

\aliascntresetthe{convention}

\newaliascnt{construction}{theorem}

\aliascntresetthe{construction}

\newaliascnt{setup}{theorem}
\newtheorem{setup}[setup]{Setup}
\aliascntresetthe{setup}

\newaliascnt{algorithm}{theorem}

\aliascntresetthe{algorithm}

\newaliascnt{observation}{theorem}

\aliascntresetthe{observation}

\newaliascnt{defprop}{theorem}

\aliascntresetthe{defprop}

\def\equationautorefname~#1\null{(#1)\null}
\def\sectionautorefname~#1\null{Section #1\null}
\def\subsectionautorefname~#1\null{\S #1\null}
\def\trdeg{{\rm trdeg}}

\begin{document}

\title{When are multidegrees positive?}

\author[Federico Castillo]{Federico Castillo}
\address[Castillo]{Max Planck Institute for Mathematics in the Sciences, Inselstra\ss e 22, 04103 Leipzig, Germany}
\email{Federico.Castillo@mis.mpg.de}

\author[Yairon Cid-Ruiz]{Yairon Cid-Ruiz}
\address[Cid-Ruiz]{Max Planck Institute for Mathematics in the Sciences, Inselstra\ss e 22, 04103 Leipzig, Germany}
\email{cidruiz@mis.mpg.de}
\urladdr{https://ycid.github.io}

\author[Binglin Li]{Binglin Li}
\address[Li]{Department of Statistics, 310 Herty Drive, University of Georgia, Athens, GA 30602, USA}
\email{binglinligeometry@uga.edu
}

\author[Jonathan Monta\~no]{Jonathan Monta\~no$^{*}$}
\address[Monta\~no ]{Department of Mathematical Sciences  \\ New Mexico State University  \\PO Box 30001\\Las Cruces, NM 88003-8001}
\thanks{$^{*}$ The fourth author is  supported by NSF Grant DMS \#2001645.}
\email{jmon@nmsu.edu}

\author[Naizhen Zhang]{Naizhen Zhang}
\address[Zhang]{Institut f\"ur Differentialgeometrie, Leibniz Universit\"at Hannover, Welfengarten 1, 30167 Hannover, Germany}
\email{naizhen.zhang@math.uni-hannover.de}

\date{\today}

\begin{abstract}
	Let $\kk$ be an arbitrary field, $\PP = \PP_\kk^{m_1} \times_\kk \cdots \times_\kk \PP_\kk^{m_p}$ be a multiprojective space over $\kk$, and $X \subseteq \PP$ be a closed subscheme of $\PP$.
	We provide necessary and sufficient conditions for the positivity of the multidegrees of $X$.  
	As a consequence of our methods, we show that when $X$ is irreducible, the support of multidegrees forms a discrete algebraic polymatroid. 
	In  algebraic terms, we characterize the positivity of the mixed multiplicities of  a standard multigraded algebra over an Artinian local ring, and we apply this  to the positivity of  mixed multiplicities of ideals. 
	Furthermore, we  use our results to recover several results in the literature in the context of combinatorial algebraic geometry. 
\end{abstract}

\maketitle

\section{Introduction}

Let $\kk$ be an arbitrary field, $\PP = \PP_\kk^{m_1} \times_\kk \cdots \times_\kk \PP_\kk^{m_p}$ be a multiprojective space over $\kk$, and $X \subseteq \PP$ be a closed subscheme of $\PP$.
The \emph{multidegrees} of $X$ are fundamental invariants that describe 
algebraic and geometric properties of $X$.
For each $\bn=(n_1,\ldots,n_p) \in \NN^p$ with $n_1+\cdots+n_p=\dim(X)$  one can define the \textit{multidegree of $X$ of type $\bn$ with respect to $\PP$}, denoted by $\deg_\PP^\bn(X)$, in different ways (see \autoref{def_multdeg}, \autoref{rem_chow_ring} and \autoref{rem_Hilb_series}). In classical geometrical terms, when $\kk$ is algebraically closed, $\deg_\PP^\bn(X)$ equals the number of points (counting multiplicity) in the intersection of $X$ with the product $L_1 \times_\kk \cdots \times_\kk L_p  \subset \PP$, where $L_i \subset \PP_\kk^{m_i}$ is a general linear subspace of dimension $m_i-n_i$ for each $1 \le i \le p$.

\medskip

The study of multidegrees goes back to pioneering work by van~der~Waerden \cite{VAN_DER_WAERDEN}.
From a more algebraic point of view, multidegrees receive the name of \emph{mixed multiplicities} (see \autoref{def_multdeg}). 
More recent papers where the notion of multidegree (or mixed multiplicity) is studied are, e.g.,  \cite{Bhattacharya,VERMA_BIGRAD,HERMANN_MULTIGRAD,TRUNG_POSITIVE,MIXED_MULT, EXPONENTIAL_VARIETIES,KNUTSON_MILLER, michalek2020maximum, CONCA_DENEGRI_GORLA}. 

\medskip

The main goal of this paper is to answer the following fundamental question considered by Trung \cite{TRUNG_POSITIVE} and by Huh \cite{Huh12} in the case $p=2$.
\begin{itemize}
	\item \emph{For $\bn \in \NN^p$ with $n_1+\cdots+n_p=\dim(X)$, when do we have that $\deg_\PP^\bn(X) >0$?}
\end{itemize}
Our main result says that the positivity of $\deg_\PP^\bn(X)$ is determined by the dimensions of the images of the natural projections from $\PP$ restricted to the irreducible components of $X$.
First, we set a basic notation: for each $\fJ = \{j_1,\ldots,j_k\} \subseteq \{1,\ldots,p\}$, let $\Pi_\fJ$ be the natural projection 
$$
\Pi_\fJ: \PP = \PP_\kk^{m_1} \times_\kk \cdots \times_\kk \PP_\kk^{m_p} \;\rightarrow\; \PP_\kk^{m_{j_1}} \times_\kk \cdots \times_\kk \PP_\kk^{m_{j_k}}.
$$
The following is the main theorem of this article. Here, we give necessary and sufficient conditions for the positivity of multidegrees.

\begin{headthm}[\autoref{thm_main_irreducible}, \autoref{thm_main}]
	\label{thmA}
	Let $\kk$ be an arbitrary field, $\PP = \PP_\kk^{m_1} \times_\kk \cdots \times_\kk \PP_\kk^{m_p}$ be a multiprojective space over $\kk$, and $X \subseteq \PP$ be a closed subscheme of $\PP$.
	Let $\bn=(n_1,\ldots,n_p) \in \NN^p$ be such that $n_1+\cdots+n_p=\dim(X)$.
	Then, $\deg_\PP^\bn(X) > 0$ if and only if there is an irreducible component $Y \subseteq X$ of $X$ that satisfies the following two conditions:
	\begin{enumerate}[\rm (a)]
		\item $\dim(Y) = \dim(X)$.
		\item For each $\fJ = \{j_1,\ldots,j_k\} \subseteq \{1,\ldots,p\}$ the inequality 
		$$
		n_{j_1} + \cdots + n_{j_k} \le \dim\big(\Pi_\fJ(Y)\big)
		$$
		holds.
	\end{enumerate} 
\end{headthm}

When $\kk$ is the field of complex numbers \autoref{thmA} is essentially covered by the geometric results in \cite[Theorems 2.14, 2.19]{KAVEH_KHOVANSKII},\footnote{In \autoref{rem_Kaveh_Kho} we briefly discuss how (over the complex numbers) \autoref{thmA} can be obtained by using the results in \cite[\S 2.2]{KAVEH_KHOVANSKII}.}
 however their methods do not extend to arbitrary fields. Here we follow an algebraic approach that allows us to prove the result for all fields, and hence a general version for algebras over Artinian local rings (see \autoref{corB}). 
The main idea in the proof of \autoref{thmA} is the study of the dimensions of the images of the natural projections after cutting by a general hyperplane (see \autoref{thm_projections}). 

\medskip

We note  that if $p=2$ and $X$ is arithmetically Cohen-Macaulay,  the conclusion of \autoref{thmA} in the irreducible case also holds for $X$ (see \cite[Corollary 2.8]{TRUNG_POSITIVE}). 
In \autoref{ex_cm} we show that this is not necessarily true for $p>2$.

\medskip
If $X$ is irreducible, then the function $r:2^{\{1,\ldots,p\}}\rightarrow \ZZ$ defined by $r(\fJ):=\dim\big(\Pi_\fJ(Y)\big)$ is a submodular function, i.e., $r(\fJ_1\cap \fJ_2)+r(\fJ_1\cup \fJ_2)\leq r(\fJ_1)+r(\fJ_2)$ for any two subsets $\fJ_1,\fJ_2\subseteq \{1,\ldots,p\}$, as proved in \autoref{prop_poly} (see also \autoref{def_algebraic_polymatroid}). 
By the Submodular Theorem (see, e.g., \cite[Theorem 3.11]{NESTED} or \cite[Appendix B]{YU})  and the inequalities of \autoref{thmA}, the points $\bn\in \NN^p$ for which $\deg_\PP^\bn(X) > 0$ are the lattice points of a {\it generalized permutohedron}. 
Defined by A.~Postnikov in \cite{POSTNIKOV} generalized permutohedra are polytopes obtained by deforming usual permutohedra. 
In recent years this family of polytopes has been studied in relation to other fields such as probability, combinatorics, and  representation theory (see \cite{YU,YONG,FACES}).

\medskip

In a more algebraic flavor, we state the translation of \autoref{thmA} to the mixed multiplicities  of a standard multigraded algebra over an Artinian local ring (see \autoref{def_mixed_mult}). 

\begin{headthm}[\autoref{cor_positive_mixed_mult}]
	\label{corB}
	Let $A$ be an Artinian local ring and $R$ be a finitely generated standard $\NN^p$-graded $A$-algebra.
	For each $1 \le j \le p$, let  $\mm_j \subset R$ be the ideal generated by the elements of degree $\ee_j$, where $\ee_j \in \NN^p$ denotes the $j$-th elementary vector.
	Let $\nn = \mm_1 \cap \cdots \cap \mm_p \subset R$.
	Let $\bn=(n_1,\ldots,n_p) \in \NN^p$ be such that $n_1+\cdots+n_p=\dim\left(R/\left(0:_R\nn^\infty\right)\right)-p$.
	Then, $e(\bn;R) > 0$ if and only if there is a minimal prime ideal $\fP \in \Min\left(0:_R\nn^\infty\right)$ of $\left(0:_R\nn^\infty\right)$ that satisfies the following two conditions:
	\begin{enumerate}[\rm (a)]
		\item $\dim\left(R/\fP\right)  = \dim\left(R/\left(0:_R\nn^\infty\right)\right)$.
		\item For each $\fJ = \{j_1,\ldots,j_k\} \subseteq \{1,\ldots,p\}$ the inequality 
		$$
		n_{j_1} + \cdots + n_{j_k} \le \dim\left(\frac{R}{\fP+\sum_{j \not\in \fJ} \mm_j}\right) - k
		$$
		holds. 
	\end{enumerate} 
\end{headthm}

For a given finite set of  ideals in a Noetherian local ring, such that  one of   them is  zero-dimensional, we can define their mixed multiplicities by considering a certain associated standard multigraded algebra (see \cite{TRUNG_VERMA_MIXED_VOL} for more information). 
These multiplicities have a long history of interconnecting problems from commutative algebra, algebraic geometry, and combinatorics, with applications to the topics of Milnor numbers, mixed volumes, and integral dependence (see, e.g., \cite{Huh12,huneke2006integral,TRUNG_VERMA_MIXED_VOL,teissier1973cycles}). 
As a direct consequence of \autoref{corB} we are able to give a characterization for the positivity of mixed multiplicities of ideals (see \autoref{cor_mixed_mult_ideals}). In another related result, we focus on homogeneous ideals generated in one degree; this case is of particular importance due to its relation with rational maps between projective varieties. 
In this setting, we provide more explicit  conditions for positivity in terms of the analytic spread of products of these ideals   (see \autoref{thm_equigen_ideals}).

\medskip

Going back to the setting of \autoref{thmA}, we switch our attention to the following discrete set 
$$
		\msupp_\PP(X) \,=\, \big\{\bn\in \NN^p \;\mid\; \deg_\PP^\bn(X)>0\big\},
$$	
which we call the {\it support of $X$ with respect to $\PP$}.  
When $X$ is irreducible,  we show that $\msupp_\PP(X)$ is  a {\it (discrete) polymatroid} (see \autoref{sub_Poly}, \autoref{prop_poly}). The latter result was included in an earlier version of this paper  when $\kk$ is algebraically closed, and an alternative proof is given by  Br{\"a}nd{\'e}n and Huh in \cite[Corollary 4.7]{LORENTZ} using the theory of Lorentzian polynomials.   An advantage of our approach is that we can describe the corresponding rank submodular functions of the polymatroids, a fact that we exploit in the applications of \autoref{sec_comb}.
Additionally, our results are valid when $X$ is just irreducible and not necessarily geometrically irreducible over $\kk$ (i.e., we do not need to assume that $X \times_\kk \overline{\kk}$ is irreducible for an algebraic closure $\overline{\kk}$ of $\kk$); it should be noticed that this generality is not covered by the statements in \cite{LORENTZ} and \cite{KAVEH_KHOVANSKII}.

\medskip

Discrete polymatroids \cite{HERZOG_HIBI_MONOMIALS} have also been studied  under the  name of M-convex sets \cite{MUROTA}. Polymatroids can also be described as the  integer points in a generalized permutohedron \cite{POSTNIKOV}, so  they are closely related to submodular functions, which are well studied in optimization, see \cite{LOVASZ} and \cite[Part IV]{SCHRIJVER} for comprehensive surveys on submodular functions, their applications, and their history.
There are  two distinguishable types of  polymatroids, linear and algebraic polymatroids, whose main properties are inherited by  their representation  in terms of other algebraic structures. 
\autoref{thmA} allows us to define another type of polymatroids, that we call \emph{Chow polymatroids}, and which interestingly lies in between the other two.
In the following theorem we summarize our main results in this direction.


\begin{headthm}[\autoref{thm_classification}]\label{thmC}
	Over an arbitrary field $\kk$, we have the following inclusions of families of polymatroids	
	$$
	\Big(\texttt{Linear polymatroids}\Big) \;\subseteq\; \Big(\texttt{Chow polymatroids}\Big) \;\subseteq\; \Big(\texttt{Algebraic polymatroids}\Big).
	$$	
	Moreover, when $\kk$ is a field of characteristic zero, the three families coincide.
\end{headthm}
If $\kk$ has positive characteristic, then these types of polymatroids do not agree. In fact, there exist examples of polymatroids which are algebraic over any field of positive characteristic but never linear (see \autoref{rem_Alg_not_line}). 

\smallskip

\autoref{thmA} can be applied to particular examples of varieties coming from combinatorial algebraic geometry. 
In \autoref{subsect_Schubert} we do so to matrix Schubert varieties; in this case the multidegrees are the coefficients of Schubert polynomials, thus our results allow us to give an alternative proof to a recent conjecture regarding the support of these polynomials (see \autoref{thm_Schubert}). 
In \autoref{sec_flag} and \autoref{sec_M0} we study certain embeddings of flag varieties and of the moduli space $\overline{M}_{0,p+3}$, respectively (see \autoref{prop_sottile} and \autoref{prop_parking}). In \autoref{sec_mixed} we recover a well-known characterization for the positivity of mixed volumes of convex bodies (see \autoref{thm_mixed}).

\smallskip

We now outline the contents of the article. 
In \autoref{section_notations} we set up the notation used throughout the document. 
We also include key preliminary definitions and results, paying special attention to the connection between mixed multiplicities of standard multigraded graded algebras and  multidegrees of their corresponding schemes.  
\autoref{sec_main}  is devoted to the proof of \autoref{thmA} and \autoref{corB}. 
Our results for mixed multiplicities of ideals are included in \autoref{sec_mix_ideals}. 
In \autoref{sec_polym} we relate our results to the theory of polymatroids. 
In particular, we show the proof of  \autoref{thmC}. 
We finish the paper with \autoref{sec_comb} where the applications to combinatorial algebraic geometry are presented.

\smallskip

We conclude the Introduction with an illustrative example.
The following example is constructed following the same ideas in \autoref{prop_linear}.

\begin{example}
Consider the polynomial ring $S=\kk[v_1,v_2,v_3][w_1,w_2,w_3]$ with the $\NN^3$-grading  $\deg(v_{i})=(0,0,0)$, $\deg(w_{i})=\ee_i$ for $1 \le i \le 3$.
Let $T$ be the $\NN^3$-graded polynomial ring 
$
T = \kk\left[x_0,\ldots,x_3\right]\left[y_0,\ldots,y_3\right]\left[z_0,\ldots,z_3\right]
$
where $\deg(x_i)=\ee_1$, $\deg(y_i)=\ee_2$ and $\deg(z_i)=\ee_3$.
Consider the $\NN^3$-graded $\kk$-algebra homomorphism 
$$
\varphi = T \rightarrow S, \qquad 
\begin{array}{llll}
x_0 \mapsto  w_1, & x_1 \mapsto v_1w_1, & x_2 \mapsto v_1w_1, & x_3 \mapsto v_1w_1, \\
y_0 \mapsto  w_2, & y_1 \mapsto v_1w_2, & y_2 \mapsto v_2w_2, & y_3 \mapsto (v_1+v_2)w_2,  \\
z_0 \mapsto  w_3, & z_1 \mapsto v_1w_3, & z_2 \mapsto v_2w_3, & z_3 \mapsto v_3w_3. 
\end{array}
$$
Note that $\fP = \Ker(\varphi) \subset T$ is an $\NN^3$-graded prime ideal. 
Let $Y \subset \PP = \PP_\kk^3 \times_\kk \PP_\kk^3 \times_\kk \PP_\kk^3$ be the closed subscheme corresponding to $\fP$.
In this case, one can easily compute the dimension of the projections $\Pi_{\fJ}(Y)$ for each $\fJ \subseteq \{1,2,3\}$, and so \autoref{thmA} implies that $\msupp_{\PP}(Y)$ is given by all $\bn=(n_1,\ldots,n_3) \in \NN^3$ satisfying the following conditions:
\begin{align*}
&n_1+n_2+n_3 = 3=\dim(Y),\\ 
&n_1+n_2 \leq 2=\dim\left(\Pi_{\{1,2\}}(Y)\right), \;\;
n_1+n_3 \leq 3=\dim\left(\Pi_{\{1,3\}}(Y)\right), \;\;
n_2+n_3 \leq 3=\dim\left(\Pi_{\{2,3\}}(Y)\right), \\
&n_1 \leq 1=\dim\left(\Pi_{\{1\}}(Y)\right), \;\;
n_2 \leq 2=\dim\left(\Pi_{\{2\}}(Y)\right), \;\;
n_3 \leq 3=\dim\left(\Pi_{\{3\}}(Y)\right).
\end{align*}
Hence $\msupp_{\PP}(Y)=\{(0,0,3),(0,1,2),(0,2,1),(1,0,2),(1,1,1)\}\subset\NN^3$. 
This set can also be represented graphically as follows:
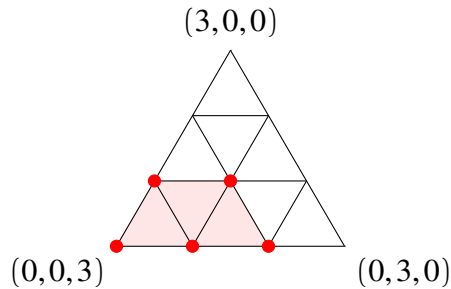
\newcommand*\rows{3}
\begin{center}
	\begin{tikzpicture}
	\draw [red!10, fill=red!10] ($(0,0)$) -- ($(2,0)$) -- ($(3/2, {1*sqrt(3)/2})$)--($(1/2, {1*sqrt(3)/2})$)-- cycle;
	
	\foreach \row in {0, 1, ...,\rows} {
		\draw ($\row*(0.5, {0.5*sqrt(3)})$) -- ($(\rows,0)+\row*(-0.5, {0.5*sqrt(3)})$);
		\draw ($\row*(1, 0)$) -- ($(\rows/2,{\rows/2*sqrt(3)})+\row*(0.5,{-0.5*sqrt(3)})$);
		\draw ($\row*(1, 0)$) -- ($(0,0)+\row*(0.5,{0.5*sqrt(3)})$);
	}
	
	\node [below left] at (0,0) {$(0,0,3)$};
	\node [below right] at (3,0) {$(0,3,0)$};
	\node [above] at ($(3/2, {3*sqrt(3)/2})$) {$(3,0,0)$};
	
	\draw [red,fill] (0,0) circle [radius = 0.08];
	
	\draw [red,fill] (1,0) circle [radius = 0.08];
	
	\draw [red,fill] (2,0) circle [radius = 0.08];
	
	\draw [red,fill] ($(1/2, {1*sqrt(3)/2})$) circle [radius = 0.08];
	
	\draw [red,fill] ($(3/2, {1*sqrt(3)/2})$) circle [radius = 0.08];
	\end{tikzpicture}
	\label{fig:example}
\end{center}	

Additionally, by using {\tt Macaulay2} \cite{MACAULAY2} we can compute that its \emph{multidegree polynomial} (see \autoref{defMsupp}) is equal to:
\[
\deg_{\PP}(Y;t_1,t_2,t_3)=\,{t}_1^{3}{t}_2^{3}+\,{t}_1^{3}{t}_2^{2}{t}_3+\,{t}_1^{3}{t}_2{t}_3^{2}+\,{t}_1^{2}{t}_2^{3}{t}_3+\,{t}_1^{2}{t}_2^{2}{t}_3^{2}.
\]
We note that here we are following the convention that $\msupp_{\PP}(Y)$ is given by the complementary degrees of the polynomial $\deg_{\PP}(Y;t_1,t_2,t_3)$; for instance, the term ${t}_1^{3}{t}_2^{3}$ corresponds to the point $(3,3,3)-(3,3,0)=(0,0,3) \in \msupp_{\PP}(Y)$.	
\end{example}


\section{Notation and Preliminaries}
\label{section_notations}

In this section, we set up the notation that is used throughout the paper. 
We also present some preliminary results needed in the proofs of our main theorems.

Let $p \ge 1$ be a positive integer. 
If $\bn = (n_1,\ldots,n_p),\bm = (m_1,\ldots,m_p) \in \ZZ^p$ are two multi-indexes, we write $\bn \ge \bm$ whenever $n_i \ge m_i$ for all $1 \le i \le p$, and $\bn > \bm$ whenever $n_j > m_j$ for all $1 \le j \le p$.
For each $1 \le i \le p$, let $\ee_i \in \NN^p$ be the $i$-th elementary vector $\ee_i=\left(0,\ldots,1,\ldots,0\right)$.
Let $\mathbf{0} \in \NN^p$ and $\mathbf{1} \in \NN^p$ be the vectors $\mathbf{0}=(0,\ldots,0)$ and $\mathbf{1}=(1,\ldots,1)$ of $p$ copies of $0$ and $1$, respectively.  
For any $\bn = (n_1,\ldots,n_p) \in \ZZ^p$, we define its weight as $\lvert \bn \rvert:= n_1+\cdots+n_p$. 
Let $[p]$ denote the set $[p]:=\{1,\ldots,p\}$.

For clarity of exposition we first introduce the main concepts in the theory of multidegrees over an arbitrary field. 
Later, we also work over Artinian local rings; we highlight important details in this more general setting in \autoref{sub_Art}.

\subsection{The case over a field}

We begin by introducing a general setup for \autoref{thmA} and its preparatory results.

\begin{setup}
	\label{setup_initial}
	Let $\kk$ be an arbitrary field.
	Let $R$ be a finitely generated standard $\NN^p$-graded algebra over $\kk$, that is,  $\left[R\right]_{\mathbf{0}}=\kk$ and $R$ is finitely generated over $\kk$ by elements of degree $\ee_i$ with $1 \le i \le p$.
	For each subset $\mathfrak{J} = \{j_1,\ldots,j_k\} \subseteq [p] = \{1, \ldots, p\}$ denote by $R_{(\fJ)}$ the standard $\NN^k$-graded $\kk$-algebra given by 
	$$
	R_{(\fJ)} := \bigoplus_{\substack{i_1\ge 0,\ldots, i_p\ge 0\\ i_{j} = 0 \text{ if } j \not\in \fJ}} {\left[R\right]}_{(i_1,\ldots,i_p)};
	$$
	for instance, for each $1 \le j \le p$, $R_{({j})}$ denotes the standard $\NN$-graded $\kk$-algebra 
	$
	R_{(j)} := \bigoplus_{k \ge 0} {\left[R\right]}_{k \cdot \ee_j}.
	$
	For each $1 \le j \le p$, let  $\mm_j \subset R$ be the ideal $\mm_j := \left([R]_{\ee_j
	}\right)$.
	Let $\nn \subset R$ be the multigraded irrelevant ideal $\nn := \mm_1 \cap \cdots \cap \mm_p$.
	For each $\fJ \subseteq [p]$, let $\nn_\fJ \subset R_{(\fJ)}$ be the corresponding multigraded irrelevant ideal $\nn_\fJ := \left(\bigcap_{j \in \fJ}\mm_j\right) \cap R_{(\fJ)}$.
	Let $X$ be the multiprojective scheme $X := \multProj(R)$ (see \autoref{def_multproj} below) and $X_\fJ$ be the multiprojective scheme $X_\fJ := \multProj(R_{(\fJ)})$ for each $\fJ \subseteq [p]$.
	To avoid trivial situations, we always assume that $X \neq \emptyset$.
\end{setup}

\begin{definition}
	\label{def_multproj}
	The multiprojective scheme  $\multProj(R)$ is given by $
	\multProj(R) := \big\{ \fP \in \Spec(R) \mid \fP \text{ is $\NN^p$-graded and } \fP \not\supseteq \nn \big\},
	$
	and its scheme structure is obtained by using multi-homogeneous localizations (see, e.g., \cite[\S 1]{HYRY_MULTIGRAD}). 
\end{definition}

The inclusion $R_{(\fJ)} \hookrightarrow R$ induces the natural projection 
\begin{align*}
\Pi_\fJ: X \rightarrow X_\fJ, \quad \fP \in X \mapsto \fP \cap R_{(\fJ)} \in X_\fJ. 
\end{align*}
We embed $X$ as a closed subscheme of a multiprojective space $\PP:=\PP_\kk^{m_1} \times_\kk \cdots \times_\kk \PP_\kk^{m_p}$.
Then, for each $\fJ = \{j_1,\ldots,j_k\} \subseteq [p]$, $\Pi_\fJ: X \rightarrow X_\fJ$ corresponds with the restriction to $X$ and to $X_\fJ$ of the natural projection 
$$
\Pi_\fJ: \PP \;\rightarrow\; \PP_\kk^{m_{j_1}} \times_\kk \cdots \times_\kk \PP_\kk^{m_{j_k}},
$$
and $X_\fJ$ becomes a closed subscheme of $\PP_\kk^{m_{j_1}} \times_\kk \cdots \times_\kk \PP_\kk^{m_{j_k}}$.

For any multi-homogeneous element $x \in R$, the closed subscheme $\multProj(R/xR) \subseteq X$ is denoted by $X \cap V(x)$.

\begin{notation}
	From now on, $\fJ=\{j_1,\ldots,j_k\}$ denotes a subset of $[p]$.
	Set $r := \dim(X)$ and $r(\fJ):=\dim\left(\Pi_\fJ(X)\right)$ for each $\fJ \subseteq [p]$.
	For a singleton set $\{i\} \subseteq [p]$, $r(\{i\})$ and $\Pi_{\{i\}}$ are simply denoted by $r(i)$ and $\Pi_i$, respectively.
\end{notation}

Note that the image of $\Pi_\fJ : X \rightarrow X_\fJ$ can be described by the following isomorphism
\begin{equation}
\label{eq_isom_0}
\Pi_\fJ(X) \;\cong\; \multProj\left(\frac{R_{(\fJ)}}{R_{(\fJ)} \cap \left(0:_R\nn^\infty\right)}\right).
\end{equation}

\begin{remark}
	\label{rem_isom_quotient_grad}
	Since $R = R_{(\fJ)} \oplus \left(\sum_{j \not\in \fJ} \mm_j\right)$, we obtain a natural isomorphism 
	$
	R_{(\fJ)} \cong \frac{R}{\sum_{j \not\in \fJ} \mm_j}
	$ 
	of $\NN^{k}$-graded $\kk$-algebras where $k = \lvert \fJ \rvert$.
\end{remark}

We now provide some preparatory results.

\begin{lemma}
	\label{lem_basics_projections}
	Under \autoref{setup_initial}, the following statements hold:
	\begin{enumerate}[\rm (i)]
		\item $r = \dim(X) = \dim\left(R/\left(0:_R\nn^\infty\right)\right) - p$.
		\item There is an isomorphism
		\begin{equation}
			\label{eq_isom}
			\Pi_\fJ(X) \;\cong\; \multProj\left(R/\Big(\left(0:_R\nn^\infty\right) + \sum_{j \not\in \fJ} \mm_j\Big)\right).
		\end{equation}
		\item If $\left(0:_R\nn^\infty\right)=0$, then $\Pi_\fJ(X) \cong \multProj(R_{(\fJ)})=X_\fJ$.
		\item If $\left(0:_R\nn^\infty\right)=0$, then $\left(0:_{R_{(\fJ)}} \nn_\fJ^\infty\right)=0$.
	\end{enumerate}
\end{lemma}
\begin{proof}
	(i) This formula follows from \cite[Lemma 1.2]{HYRY_MULTIGRAD} (also, see \cite[Corollary 3.5]{SPECIALIZATION_RAT_MAPS}).
	
	(ii) From the natural maps $R_{(\fJ)} \hookrightarrow R \twoheadrightarrow R/\left(0:_R\nn^\infty\right)$, we obtain a natural isomorphism 
	$$
	R_{(\fJ)}/\left(R_{(\fJ)} \cap \left(0:_R\nn^\infty\right)\right) \xrightarrow{\cong} \big(R/\left(0:_R\nn^\infty\right)\big)_{(\fJ)}.
	$$
	By using \autoref{rem_isom_quotient_grad} it follows that $\big(R/\left(0:_R\nn^\infty\right)\big)_{(\fJ)} \cong R/\left(\left(0:_R\nn^\infty\right) + \sum_{j \not\in \fJ} \mm_j\right)$.
	Therefore, the claimed isomorphism is obtained from \autoref{eq_isom_0}.
	
	(iii) It follows directly from part (ii) and \autoref{rem_isom_quotient_grad}.
	
	(iv) This part is clear.	
\end{proof}

\smallskip

Let $P_R(\ttt)=P_R(t_1,\ldots,t_p) \in \QQ[\ttt]=\QQ[t_1,\ldots,t_p]$ be the \emph{Hilbert polynomial} of $R$ (see, e.g., \cite[Theorem 4.1]{HERMANN_MULTIGRAD}, \cite[Theorem 3.4]{MIXED_MULT}).
Then, the degree of $P_R$ is equal to $r$ and 
$$
P_R(\nu) = \dim_\kk\left([R]_\nu\right) 
$$
for all $\nu \in \NN^p$ such that $\nu \gg \mathbf{0}$.
Furthermore, if we write 
\begin{equation}
\label{eq_Hilb_poly}
P_{R}(\ttt) = \sum_{n_1,\ldots,n_p \ge 0} e(n_1,\ldots,n_p)\binom{t_1+n_1}{n_1}\cdots \binom{t_p+n_p}{n_p},
\end{equation}
then $0 \le e(n_1,\ldots,n_r) \in \ZZ$ for all $n_1+\cdots+n_p = r$.

\begin{remark}\label{modSat} The following are basic properties of Hilbert polynomials.

\begin{enumerate}[(i)]
\item Since $\nn^k(0:_R\nn^\infty)=0$ for $k\gg 0$ we have $\dim_\kk\left([R]_\nu\right) = \dim_\kk\left([R/(0:_R\nn^\infty)]_\nu\right) $ for $\nu\gg \mathbf{0}$. Thus, $$P_{R}(\ttt)=P_{R/(0:_R\nn^\infty)}(\ttt).$$

\item Let $\LL$ be a field extension of $\kk$. 
Then, $R \otimes_\kk \LL$  is a finitely generated standard $\NN^p$-graded $\LL$-algebra and   
$
\dim_{\LL} \left([R \otimes_\kk \LL]_\nu\right)  = \dim_\kk\left([R]_\nu\right)
$
for all $\nu \in \NN^p$. Thus,  $$P_{R \otimes_\kk \LL}(\ttt) =P_{R}(\ttt).$$  
In particular, one can always assume $\kk$ is an infinite field (for instance, we can substitute $\kk$ by a purely transcendental field extension $\kk(\xi)$).
\end{enumerate}

\end{remark}

Under the notation of \autoref{eq_Hilb_poly} we define the following invariants. 

\begin{definition} \label{def_multdeg}
Let $\bn = (n_1,\ldots,n_p) \in \NN^p$ with $\lvert\bn\rvert=r$. Then:
	\begin{enumerate}[(i)]
		\item $e(\bn,R) := e(n_1,\ldots,n_p)$ is the \textit{mixed multiplicity of $R$ of type $\mathbf{n}$}.
		\item $\deg_\PP^\bn(X):=e(n_1,\ldots,n_p)$ is the \textit{multidegree of $X=\multProj(R)$ of type $\bn$ with respect to $\PP$}. 
	\end{enumerate} 
\end{definition}

As stated in the Introduction, in classical geometrical terms, when $\kk$ is algebraically closed, $\deg_\PP^\bn(X)$ is also equal to the number of points (counting multiplicity) in the intersection of $X$ with the product $L_1 \times_\kk \cdots \times_\kk L_p  \subset \PP$, where $L_i \subseteq \PP_\kk^{m_i}$ is a general linear subspace of dimension $m_i-n_i$ for each $1 \le i \le p$ (see \cite{VAN_DER_WAERDEN}, \cite[Theorem 4.7]{MIXED_MULT}).

The multidegrees of $X$ can be defined easily in terms of Chow rings and in terms of Hilbert series.

\begin{remark} 
	\label{rem_chow_ring} The Chow ring of $\PP = \PP_\kk^{m_1} \times_\kk \cdots \times_\kk \PP_\kk^{m_p}$ is given by 
	$$
	A^*(\PP) = \frac{\ZZ[H_1,\ldots,H_p]}{\left(H_1^{m_1+1},\ldots,H_p^{m_p+1}\right)}
	$$
	where $H_i$ represents the class of the inverse image of a hyperplane of $\PP_\kk^{m_i}$ under the natural projection $\Pi_i: \PP \rightarrow \PP_\kk^{m_i}$. 
	Then, the class of the cycle associated to $X$ coincides with
	$$	\left[X\right] \;=\; \sum_{\substack{0 \le n_i \le m_i\\ \lvert\bn\rvert=r}} \deg_\PP^\bn\left(X\right)\, H_1^{m_1-n_1}\cdots H_p^{m_p-n_p} \;\in A^*(\PP).	$$	
\end{remark}

\begin{remark}
	\label{rem_Hilb_series}
	By considering the Hilbert series 
	$
	\Hilb_R(t_1,\ldots,t_p) := \sum_{\nu \in \NN^p} \dim_\kk\left([R]_\nu\right)t_1^{\nu_1}\cdots t_p^{\nu_p} 
	$
	of $R$, one can analogously define the notions of mixed multiplicities and multidegrees  (see \cite[\S 8.5]{MILLER_STURMFELS}, \cite[Theorem A]{MIXED_MULT}).
	Here we quickly derive this analogous definition because we shall use it in \autoref{subsect_Schubert}.
	Let $S=\kk[x_{1,0},x_{1,1},\ldots,x_{1,m_1}] \otimes_\kk \cdots \otimes_\kk \kk[x_{p,0},x_{p,1},\ldots,x_{p,m_p}]$ be the multigraded polynomial ring corresponding with $\PP=\PP_\kk^{m_1} \times_\kk \cdots \times_\kk \PP_\kk^{m_p}$, that is $\PP = \multProj(S)$.
	By considering an $S$-free resolution of $R$, we can write 
	$$
	\Hilb_R(\ttt)=\mathcal{K}(R;\ttt)/\mathbf{(1-t)^{m+1}} = \mathcal{K}(R;t_1,\ldots,t_p)/\prod_{i=1}^p(1-t_i)^{m_i+1},
	$$ 
	where $\mathcal{K}(R;\ttt)$ is called the \emph{K-polynomial of $R$} (see \cite[Definition 8.21]{MILLER_STURMFELS}).
	Let $\mathcal{C}(R;\ttt) \in \ZZ[t_1,\ldots,t_p]$ be the sum of all the terms in $\mathcal{K}(R;\mathbf{1-t})$ of total degree equal to $\dim(S)-\dim(R)$ (see \cite[Definition 8.45]{MILLER_STURMFELS}).
	Then, if $(0:_R\nn^\infty)=0$, we obtain the equality 
	$$
	\mathcal{C}(R;\ttt) = \sum_{\substack{0 \le n_i \le m_i\\ \lvert\bn\rvert=r}} \deg_\PP^\bn\left(X\right)\, t_1^{m_1-n_1}\cdots t_p^{m_p-n_p}.
	$$
\end{remark}
\begin{proof}
	From \cite[Theorem A(I)]{MIXED_MULT} we have $\Hilb_R(\ttt) = \sum_{\lvert\mathbf{k}\rvert=\dim(R)}Q_\mathbf{k}(\ttt)/\mathbf{(1-t)}^\mathbf{k}$ where $Q_\mathbf{k}(\ttt) \in \ZZ[\ttt]$.
	The assumption $(0:_R\nn^\infty)=0$ gives that $\dim(R)=r+p$ (see \autoref{lem_basics_projections}(i)).
	Hence, by using \cite[Theorem A(II,III)]{MIXED_MULT} we obtain that $\deg_\PP^\bn\left(X\right)=e(\bn;R) = Q_{\mathbf{n+1}}(\mathbf{1})$ for all $\lvert\bn\rvert = r$.
	Also, the assumption $(0:_R\nn^\infty)=0$ and \cite[Theorem 2.8(ii)]{MIXED_MULT} imply that $Q_\mathbf{k}(\mathbf{1})=0$ when $\lvert\mathbf{k}\rvert=\dim(R)$ and $\mathbf{k}_i=0$ for some $1 \le i \le p$.
	
	After writing $\Hilb_R(\ttt) = \sum_{\lvert\mathbf{k}\rvert=\dim(R)}Q_\mathbf{k}(\ttt)\mathbf{(1-t)}^\mathbf{m+1-k}/\mathbf{(1-t)}^\mathbf{m+1}$, we obtain the equality 
	$$\mathcal{K}(R;\ttt)=\sum_{\lvert\mathbf{k}\rvert=\dim(R)}Q_\mathbf{k}(\ttt)\mathbf{(1-t)}^\mathbf{m+1-k}.
	$$	
	Making the substitution $t_i\mapsto (1-t_i)$ and choosing the terms of total degree  $\dim(S)-\dim(R)=\sum_{i=1}^pm_i-r$, it follows that $\mathcal{C}(R;\ttt)=\sum_{\lvert\mathbf{k}\rvert=\dim(R)}Q_\mathbf{k}(\mathbf{1})\mathbf{t^{m+1-k}} = \sum_{\lvert\mathbf{n}\rvert=r}Q_{\mathbf{n+1}}(\mathbf{1})\mathbf{t^{m-n}}$.
	So, the result is clear.
\end{proof}

Although in the proofs of \autoref{thmA} and \autoref{corB} we do not exploit the fact that multidegrees can be defined as in \autoref{rem_chow_ring}, we do encode the multidegrees in a homogeneous polynomial that mimics the cycle associated to $X$ in the Chow ring $A^*(\PP)$. 
The following objects are the main focus of this
paper.

\begin{definition}\label{defMsupp}
	Let $X\subseteq \PP = \PP_{\kk}^{m_1} \times_\kk \cdots \times_\kk \PP_{\kk}^{m_p}$ be a closed subscheme with $r = \dim(X)$. 
	We denote the \emph{multidegree polynomial of $X$ with respect to $\PP$} as the homogeneous polynomial 
	$$
	\deg_\PP(X;t_1,\ldots,t_p) \;:=\;  \sum_{\substack{0 \le n_i \le m_i\\ \lvert\bn\rvert=r}} \deg_\PP^\bn\left(X\right)\, t_1^{m_1-n_1}\cdots t_p^{m_p-n_p} \;\in\; \NN[t_1,\ldots,t_p]
	$$
	of degree $m_1+\cdots+m_p-r$.
	We say that the \textit{support of $X$ with respect to $\PP$} is given by
	$$
		\msupp_\PP(X) \,:=\, \big\{\bn\in \NN^p \;\mid\; \deg_\PP^\bn(X)>0\big\}.
	$$	
\end{definition}

\begin{remark}
	Note that under the assumption $(0:_R\nn^\infty)=0$ we obtain the equality $\deg_\PP(X;\ttt) = \mathcal{C}(R;\ttt)$.
\end{remark}

\smallskip

\subsection{The case over an Artinian local ring}\label{sub_Art}

In this subsection, we show how the mixed multiplicities are defined for a standard multigraded algebra over an Artinian local ring.

\begin{setup}
	\label{setup_Artinian}
	Keep the notations and assumptions introduced in \autoref{setup_initial} and now substitute the field $\kk$ by an Artinian local ring $A$.
\end{setup}

In this setting, the notion of mixed multiplicities is defined essentially in the same way as in \autoref{def_multdeg}.

\begin{definition}
	\label{def_mixed_mult}
	Let $P_R(\ttt)=P_R(t_1,\ldots,t_p) \in \QQ[\ttt]=\QQ[t_1,\ldots,t_p]$ be the \emph{Hilbert polynomial} of $R$ (see, e.g., \cite[Theorem 4.1]{HERMANN_MULTIGRAD}, \cite[Theorem 3.4]{MIXED_MULT}).
	Then, as before,  the degree of $P_R$ is  equal to  $\dim\left(R/\left(0:_R\nn^\infty\right)\right) - p$ and 
	$$
	P_R(\nu) = \text{length}_A\left([R]_\nu\right) 
	$$
	for all $\nu \in \NN^p$ such that $\nu \gg \mathbf{0}$.
	If we write 
	$
	P_{R}(\ttt) = \sum_{n_1,\ldots,n_p \ge 0} e(n_1,\ldots,n_p)\binom{t_1+n_1}{n_1}\cdots \binom{t_p+n_p}{n_p},
	$
	then $0 \le e(n_1,\ldots,n_r) \in \ZZ$ for all $n_1+\cdots+n_p = \dim\left(R/\left(0:_R\nn^\infty\right)\right) - p$.
	For each $\bn = (n_1,\ldots,n_p) \in \NN^p$ with $\lvert \bn\rvert=\dim\left(R/\left(0:_R\nn^\infty\right)\right) - p$, we set that $e(\bn,R) := e(n_1,\ldots,n_p)$ is the \textit{mixed multiplicity of $R$ of type $\mathbf{n}$}.
\end{definition}

\subsection{Polymatroids}\label{sub_Poly} In this subsection we include some relevant information about polymatroids.

\begin{definition}
	Let $E$ be a finite set and $r$ a function $r:2^{E}\rightarrow \ZZ_{\geq 0}$ satisfying the following two properties: (i) it is \emph{non-decreasing}, i.e., $r(\fT_1)\leq r(\fT_2)$ if $\fT_1\subseteq \fT_2\subseteq E$, and (ii) it is \emph{submodular}, i.e., $r(\fT_1\cap \fT_2)+r(\fT_1\cup \fT_2)\leq r(\fT_1)+r(\fT_2)$ if $\fT_1,\fT_2 \subseteq E$. The function $r$ is called a {\it rank function on $E$}. We usually let $E=[p]$.
	
	A \textit{$($discrete$)$ polymatroid} $\mathcal{P}$ on $[p]$ with rank function $r$ is a collection of points in $\NN^p$ of the following form
	\[
	\mathcal{P=}\left\{\bx=(x_1,\ldots, x_p)\in \NN^p \;\mid\; \sum_{j\in\fJ} x_j\leq r(\fJ), \;\forall \fJ\subsetneq [p], \;\sum_{i\in [p]} x_i=r([p])\right\}.
	\]
	By definition, a  polymatroid consists of the integer points of a polytope (the convex hull of $\mathcal{P}$), we call that polytope a \emph{base polymatroid polytope}. We note that a polymatroid is completely determined by its rank function.
\end{definition}

\begin{remark}\label{rem:matroid}
If the rank function of $\mathcal{P}$ satisfies $r(\{i\})\leq 1$ for every $i\in [p]$, then $\mathcal{P}$ is called a \emph{matroid}. In other words, matroids are discrete polymatroids where every integer point is an element of $\{0,1\}^p$. A general reference for matroids is \cite{OXLEY}.
\end{remark}

In the following definition we consider the standard notions of linear and algebraic matroids  (see  \cite[Chapter 6]{OXLEY}) and adapt them to the polymatroid case.

\begin{definition}\label{def_algebraic_polymatroid}
	Let $\mathcal{P}$ be a polymatroid.
	\begin{itemize}
		\item 	We say $\mathcal{P}$ is \textit{linear} over a field $\kk$ if there exists a $\kk$-vector space $V$ and subspaces $V_i, i\in [p]$ such that for every $\fJ\subseteq [p]$ we have $r(\fJ)=\dim_\kk\left(\sum_{j\in\fJ} V_j\right)$ \cite[Proposition 1.1.1]{OXLEY}. The vector space $V$ together with the subspaces $V_i$ for $1\leq i\leq p$, are a \emph{linear representation} of $\mathcal{P}$.
		\item 	We  say $\mathcal{P}$ is \textit{algebraic} over a field $\kk$ if there exists a field extension $\kk\hookrightarrow \LL$ and intermediate field extensions $\LL_i, i\in [p]$ such that for every $\fJ\subseteq [p]$ we have $r(\fJ)=\text{trdeg}_\kk\left(\bigwedge_{j\in \fJ} \LL_j\right)$, 
		where $\bigwedge_{j\in \fJ} \LL_j$ is the \emph{compositum} of the subfields, i.e., the smallest subfield in $\LL$ containing all of them \cite[Theorem 6.7.1]{OXLEY}. The field $\LL$ together with the subfields $\LL_i$ for $1\leq i\leq p$, are an \emph{algebraic representation} of $\mathcal{P}$. 
	\end{itemize}
\end{definition}

\section{A characterization for the positivity of multidegrees}\label{sec_main}

In this section, we focus on characterizing the positivity of multidegrees and our main goal is to prove \autoref{thmA} and \autoref{corB}.
Throughout this section we continue using the same notations and assumptions of \autoref{section_notations}.

We begin with the following result that relates the Hilbert polynomial $P_R(\ttt) \in \QQ[\ttt]$ of $R$ with the dimensions $r(\fJ)=\dim\left(\Pi_{\fJ}(X)\right)$ of the schemes $\Pi_{\fJ}(X)$. 
It extends \cite[Theorem 1.7]{TRUNG_POSITIVE} to a multigraded setting.

\begin{proposition}
	\label{thm_deg_Hilb_pol}
	Assume \autoref{setup_initial}.
	For each $\fJ = \{j_1,\ldots, j_k\}\subseteq [p]$, let $\deg(P_R;\fJ)$ be the degree of the Hilbert polynomial $P_R$ in the variables $t_{j_1},\ldots, t_{j_k}$. 
	Then, for every such $\fJ = \{j_1,\ldots, j_k\}$ we have that
	$$
	\deg(P_R; \fJ) = r(\fJ).
	$$
\end{proposition}
\begin{proof}

We may assume that $(0:_R\nn^\infty)=0$ and $\kk$ is an infinite field by \autoref{modSat}. 
	Fix $\fJ=\{j_1,\ldots, j_k\}\subseteq [p]$ and let $w\in \NN$ be such that $\dim_\kk([R]_{\fn})=P_R(\fn)$ for every $\fn=(n_1,\ldots, n_p)\gs w\fone$.     
	Let $(d_1,\ldots. d_k)$ be such that $\delta:=\deg(P_R; \fJ)=d_1+\cdots+d_k$ and $t_{j_1}^{d_1}\cdots t_{j_k}^{d_k}$  divides a term of $P_R$. 
	
	Let $q$ be a polynomial in the variables $\{t_i\mid i\not\in \fJ \}$ such that $P_R-q\cdot t_{j_1}^{d_1}\cdots t_{j_k}^{d_k}$  has no term  divisible by $t_{j_1}^{d_1}\cdots t_{j_k}^{d_k}$. 
	Let $\fs=(s_i\mid i\not\in \fJ )\in \NN^{p-|\fJ|}$ be a vector of integers such that $\fs\gs w\fone$ and $q(\fs)\neq 0$. 
	Thus, if one evaluates $t_i=s_i$ in $P_R$ for every $i\not\in \fJ $ one obtains a polynomial $Q$ on the variables $t_{j_1},\cdots, t_{j_k}$ of degree $\delta$. 
	On the other hand, by \cite[Theorem 3.4]{MIXED_MULT}, for  $n_{j_1},\cdots, n_{j_k}\gs w$ this polynomial $Q$ coincides with the Hilbert polynomial of the $R_{(\fJ)}$-module generated by $[R]_{\fs'}$, where $\fs'_i=\fs_i$ if $i\not\in\fJ$ and  $\fs'_i=0$ otherwise. 
	Call this module $M$.
	
	Since $(0:_R\nn^\infty)=0$,  for every $1\le i\le p$ we have $\grade(\mm_i) \ge 1$, and then 
	there exist  elements $y_i \in [R]_{\ee_i}$ which are  non-zero-divisors (see, e.g., \cite[Lemma 1.5.12]{BRUNS_HERZOG}). 
	From the fact that $y_1^{\fs_1'}\cdots y_p^{\fs_p'} \in M$, it follows that $\Ann_{R_{(\fJ)}}(M)=0$.
	Therefore, $\delta= \dim\left(\Supp(M) \cap X_{(\fJ)}\right) = r(\fJ)$, by \cite[Theorem 3.4]{MIXED_MULT}, finishing the proof.
\end{proof}

In the following remark we gather some basic relations for the radicals of certain ideals.

\begin{remark}
	\label{rem_radicals}
	(i) Let $I, J, K \subset R$ be ideals. 
	If $J \subset \sqrt{K}$, then $I+J \subset \sqrt{I+K}$.
	In particular, if $\sqrt{J} = \sqrt{K}$, then $\sqrt{I+J} = \sqrt{I+K}$.
	
	\noindent
	(ii) For any element $x \in \mm_1$, since $(x:_R\mm_1^\infty)\mm_1^k \subset (x)$ for some $k >0$, it follows that $\sqrt{(x)} = \sqrt{(x:_R\mm_1^\infty)\mm_1^k} = \sqrt{(x:_R\mm_1^\infty)\cap \mm_1}$.
\end{remark}

If $\kk$ is an infinite field, then for each $1 \le i \le p$ we say that a property $\mathbf{P}$ is  satisfied by a \emph{general element} in the $\kk$-vector space $[R]_{\ee_i}$, if there exists a dense open subset $U$ of ${\left[R\right]}_{\ee_i}$ with the Zariski topology such that every element in $U$ satisfies the property $\mathbf{P}$.

The following three technical lemmas are important steps for the proof of \autoref{thm_projections}. 

\begin{lemma}
	\label{lem_radical_sats}
	Assume \autoref{setup_initial} with $\kk$ being an infinite field.
	Suppose that $R$ is a domain. 
	Let  $x \in [R]_{\ee_1}$ be a general element. 
	Then, we have the equality
	$
	\sqrt{\left(x:_{R}\nn^\infty\right)} = \sqrt{\left(x:_{R}\mm_1^\infty\right)}
	$.
\end{lemma}
\begin{proof}
	Since $(0:_R\nn^\infty)=0$, we have that $\HT(\mm_j) \ge 1$ for every $1\le j\le p$. 
	Consider the following finite set of prime ideals
	$$
	\mathfrak{S} = \big\lbrace \fP \in \Spec(R)  \mid \fP \in \Min(\mm_j) \text{ for some } 2 \le j \le p \text{ and } \fP \not\supseteq \mm_1 \big\rbrace.
	$$
	By using the Prime Avoidance Lemma and the fact that $\kk$ is infinite, for a general element $x \in [R]_{\ee_1}$ we have that $x \not\in \bigcup_{\fP \in \mathfrak{S}} \fP$.
	If $\fP \in \Min\left(x:_R\mm_1^\infty\right)$, then $\HT(\fP) \le 1$ by Krull's Principal Ideal Theorem, and so we would have that $\fP \in \mathfrak{S}$ whenever $ \fP \not\supseteq \mm_1$ and $\fP \supseteq \mm_j$ for some $2 \le j \le p$.
	Therefore, for any $\fP \in \Spec(R)$ and a general element $x \in [R]_{\ee_1}$, if $\fP \in \Min\left(x:_R\mm_1^\infty\right)$	 we get $\fP \supseteq (x:_R\nn^\infty) = (x:_R(\mm_1\cap \mm_2 \cap \cdots \cap \mm_p)^\infty)$; so,  $\sqrt{(x:_R\nn^\infty)} = \sqrt{(x:_R\mm_1^\infty)}$.
\end{proof}

The lemma below is necessary for some reduction arguments in \autoref{thm_projections}.

\begin{lemma}
	\label{lem_cut_project}
	Assume \autoref{setup_initial} with $\kk$ being an infinite field.
	Suppose that $R$ is a domain. 
	Let  $x \in [R]_{\ee_1}$ be a general element and set $Z = X \cap V(x) = \multProj(R/xR)$. 
	Then, for each $\fJ = \{1,j_2,\ldots,j_k\}\subseteq [p]$, the following statements hold:
	\begin{enumerate}[\rm (i)]		
		\item $\dim\left(\Pi_\fJ(Z)\right) = \dim\left(X_\fJ \cap V(x)\right)$, where  $X_\fJ \cap V(x) = \multProj\left(R_{(\fJ)}/xR_{(\fJ)}\right)$.
		
		\item $\dim(\Pi_\fL(Z)) = \dim\big(\Pi_\fL'(X_\fJ \cap V(x))\big)$, where $\mathfrak{L} = \fJ \setminus \{1\}$ and $\Pi_\fL'$ denotes the natural projection $\Pi_\fL' : \PP_\kk^{m_1} \times_\kk \PP_\kk^{m_{j_2}} \times_\kk \cdots \times_\kk \PP_\kk^{m_{j_k}} \rightarrow \PP_\kk^{m_{j_2}} \times_\kk \cdots \times_\kk \PP_\kk^{m_{j_k}}$.		
	\end{enumerate}
\end{lemma}
\begin{proof}
	For notational purposes, let $\bbb_\fJ := \mm_1 \cap R_{(\fJ)}$.
	
	(i) From \autoref{eq_isom} we have that
	$
	\Pi_{\fJ}(Z) \cong \multProj\big(R/\big((x:_R\nn^\infty)+\sum_{j\not\in \fJ}\mm_j\big)\big).
	$
	Since we are assuming $1 \in \fJ$, from \autoref{rem_isom_quotient_grad} we obtain the natural isomorphism
	$$
	R_{(\fJ)}/\left(x:_{R_{(\fJ)}} \bbb_\fJ^\infty\right) \;\xrightarrow{\cong} \; R/\Big((x:_R\mm_1^\infty)+\sum_{j\not\in \fJ}\mm_j\Big);
	$$
	indeed, for $\ell \ge 0$ and $y \in R_{(\fJ)}$, one notices that $\bbb_\fJ^\ell \cdot y \in xR_{(\fJ)}$ if and only if $\mm_1^\ell \cdot y \in xR$.

	By \autoref{lem_radical_sats} and  \autoref{rem_radicals}(i) we have    $\sqrt{(x:_R\mm_1^\infty)+\sum_{j\not\in \fJ}\mm_j}=\sqrt{(x:_R\nn^\infty)+\sum_{j\not\in \fJ}\mm_j}$, and  by applying \autoref{lem_radical_sats} to the ring $R_{(\fJ)}$ we obtain  $\sqrt{(x:_{R_{(\fJ)}}\bbb_\fJ^\infty)}=\sqrt{(x:_{R_{(\fJ)}}\nn_\fJ^\infty)} $. 
	It follows that 
	\begin{equation}
	\label{eq_radicals_fJ}
	R_{(\fJ)} \;\big/\; \sqrt{(x:_{R_{(\fJ)}}\nn_\fJ^\infty)} \;\cong\; R \;\big/\; \sqrt{(x:_R\nn^\infty)+\sum_{j\not\in \fJ}\mm_j},
	\end{equation}
	which gives the result.
	
	(ii) By using \autoref{eq_isom} we obtain that $\Pi_{\fL}(Z) \cong \multProj\big(R/\big((x:_R\nn^\infty)+\sum_{j\not\in \fJ}\mm_j + \mm_1\big)\big)$ and that
	$\Pi_{\fL}'\left(X_\fJ \cap V(x)\right)  \cong \multProj\left(R_{(\fJ)} /\left( (x:_{R_{(\fJ)}}\nn_\fJ^\infty) + \bbb_\fJ\right)\right)$. 
	Since the isomorphism in \autoref{eq_radicals_fJ} can be extended to 
	\begin{equation*}
	\label{eq_radicals_bJ}
	R_{(\fJ)} \;\Big/\; \left( \sqrt{(x:_{R_{(\fJ)}}\nn_\fJ^\infty)} + \bbb_\fJ\right) \;\cong\; R \;\Big/\; \left(\sqrt{(x:_R\nn^\infty)+\sum_{j\not\in \fJ}\mm_j} + \mm_1 \right),
	\end{equation*}
	the result follows from \autoref{rem_radicals}(i).
\end{proof}

We continue with the next auxiliary lemma that allows us to simplify the proof of \autoref{thm_projections}.

\begin{lemma}
	\label{lem_dim_projects}
	Assume \autoref{setup_initial} with $\kk$ being an infinite field.
	Suppose that $R$ is a domain and $r(1)\ge 1$. 
	Let  $x \in [R]_{\ee_1}$ be a general element and set $Z = X \cap V(x) = \multProj(R/xR)$. 
	Then, the following statements hold:
	\begin{enumerate}[\rm (i)]
		
		\item If $1 \in \fJ \subseteq [p]$, then $\dim\left(\Pi_\fJ(Z)\right) = r(\fJ) - 1$; in particular, $\dim(Z) = r - 1$. 
		
		\item If $1 \not\in \fJ$  and $r(\fK) > r(\fJ)$, where $\fK = \{1\} \,\cup\, \fJ \subseteq [p]$, then $\dim\left(\Pi_\fJ(Z)\right)=r(\fJ)$.		
	\end{enumerate}
\end{lemma}
\begin{proof}	
(i) First, from \autoref{lem_cut_project}(i) it suffices to compute $\dim\left(X_\fJ \cap V(x)\right)$, where  $X_\fJ \cap V(x) = \multProj\left(R_{(\fJ)}/xR_{(\fJ)}\right)$.
	For $\fJ = \{1, j_2,\ldots,j_k\} \subseteq [p]$, note that $\Pi_1(X) \cong \Pi_1'(X_\fJ)$, where $\Pi_1'$ denotes the natural projection $\Pi_1':\PP_\kk^{m_1}  \times_\kk\PP_\kk^{m_{j_2}} \times_\kk \cdots \times_\kk \PP_\kk^{m_{j_k}} \rightarrow \PP_\kk^{m_1}$.
	Therefore, neither the assumption nor the conclusion changes if we substitute $R$ by $R_{(\fJ)}$ and $X$ by $X_\fJ$, and we do so.

	From the short exact sequence 
	$$
	0 \rightarrow R(-\ee_1) \xrightarrow{x} R \rightarrow R/xR \rightarrow 0,
	$$
	we obtain $P_{R/xR}(\ttt) = P_R(\ttt) - P_R(\ttt - \ee_1)$.
	By using \autoref{thm_deg_Hilb_pol}, $\deg(P_R;t_1) = r(1) \ge 1$ and so $P_R$ is non-constant as a univariate polynomial in the variable $t_1$.
	Thus, $P_{R/xR}(\ttt) \neq 0$ which implies that $(x:_R\nn^\infty)$ is a proper ideal.
	So, Krull's Principal Ideal Theorem yields that $\HT(x:_R\nn^\infty)=1$ and that
	$$
	\dim(Z) = \dim(R/(x:_R\nn^\infty)) - p = \dim(R) - 1 - p = (r+p)-1-p=r-1.
	$$


	(ii) By using \autoref{lem_cut_project}(ii), we can substitute $R$ by $R_{(\fK)}$ and $X$ by $X_\fK$, and we do so.
	So, we may assume that $\fK = [p]$ and $\fJ = \{2,\ldots,p\}$.
	From \autoref{eq_isom} we get the isomorphism 
	\begin{equation}
	\label{eq_second_project_Z}
	\Pi_\fJ(Z) \cong \multProj\Big(R/\left((x:_R\nn^\infty)+\mm_1\right)\Big).
	\end{equation}
	The equality 
	\begin{equation}
	\label{eq_radicals_project}
	\sqrt{(x:_R\nn^\infty) + \mm_1} = \sqrt{(x:_R\mm_1^\infty) + \mm_1}
	\end{equation}
	follows from \autoref{lem_radical_sats} and \autoref{rem_radicals}(i).
	The assumption yields that 
	$$
	\HT(\mm_1)=\dim\left(R\right)-\dim\left(R_{(\fJ)}\right)=(r+p)-(r(\fJ)+p-1)\ge 2,
	$$ 
	then as a consequence Krull's Principal Ideal Theorem it follows that $\sqrt{(x)} = \sqrt{(x:_R\mm_1^\infty)}$; therefore, \autoref{rem_radicals}(i) implies that  $\sqrt{(x:_R\mm_1^\infty) + \mm_1}=\sqrt{\mm_1}$.
	By summing up, we obtain the equalities $\dim\big(R/\left((x:_R\nn^\infty)+\mm_1\right)\big)=\dim(R/\mm_1)=r(\fJ) +p-1$, and so the result follows.
	\end{proof}

The next important theorem computes the dimension of the image of the projections $\Pi_\fJ$ after cutting with a general hyperplane under certain conditions.   For the proof of this result, we need the following version of Grothendieck's Connectedness Theorem. For that, we  recall the definitions
\begin{align*}
&c(R) := \min\big\{ \dim(R/\aaa) \mid \aaa \subset R \text{ is an ideal and } \Spec(R) \setminus V(\aaa) \text{ is disconnected} \big\},\\
&\sdim(R) := \min\big\{\dim(R/\fP) \mid \fP \in \Min(R) \big\} \text{ and }\\
&\ara(\aaa) :=\min\{n \mid \sqrt{(a_1,\ldots,a_n)} = \sqrt{\aaa} \text{ and } a_i \in R \}
\end{align*}
for any ideal $\aaa \subset R$.

\begin{lemma}[{\cite[Proposition 2.1]{BRODMANN_RUNG},  \cite[Lemma 2.6]{TRUNG_POSITIVE}}]
	\label{lem_dimension_connect}
	For two proper homogeneous ideals $\aaa,\bbb \subset R$, if $\min\{\dim(R/\aaa), \dim(R/\bbb)\} > \dim(R/(\aaa+\bbb))$, then 
	$$
	\dim(R/(\aaa+\bbb)) \ge \min\{ c(R), \;\sdim(R)-1 \} - \ara(\aaa \cap \bbb).
	$$
\end{lemma}

We are now ready to present the following theorem.

\begin{theorem}
	\label{thm_projections}
	Assume \autoref{setup_initial} with $\kk$ being an infinite field. 
	Suppose that $R$ is a domain and $r(1) \ge 1$.
	Let $x \in [R]_{\ee_1}$  be a general element and set $Z = X \cap V(x) = \multProj(R/xR)$.
	Then, for each $\fJ \subseteq [p]$ we have that 
	$$
	\dim(\Pi_\fJ(Z)) = 
	\min\Big\{ r(\fJ),\; r\big(\fJ \cup \{1\}\big)-1\Big\}. 
	$$
\end{theorem}
\begin{proof}
For each $\fJ=\{j_1,\ldots,j_k\} \subseteq \fK = \{h_1,\ldots,h_\ell\} \subseteq [p]$ we have that $\Pi_\fJ(Z) =  \Pi_\fJ^\prime(\Pi_\fK(Z))$ where $\Pi_\fJ^\prime$ denotes the natural projection $\Pi_\fJ^\prime : \PP_\kk^{m_{h_1}} \times_\kk \cdots \times_\kk \PP_\kk^{m_{h_\ell}} \rightarrow \PP_\kk^{m_{j_1}} \times_\kk \cdots \times_\kk \PP_\kk^{m_{j_k}}$.	
	So, from \autoref{lem_dim_projects}(i) it follows that the inequality ``$\le$'' holds in the desired equality.
	
	Due to \autoref{lem_dim_projects}, in order to show the reversed inequality ``$\ge$'', it suffices to show that $\dim(\Pi_\fJ(Z)) \ge r(\fJ)-1$ when $1 \not\in \fJ$  and $r(\fK) = r(\fJ)$, where $\fK = \{1\} \,\cup\, \fJ \subseteq [p]$. 
	By  using \autoref{lem_cut_project}(ii), we assume may that $\fK = [p]$ and $\fJ = \{2,\ldots,p\}$.
	From \autoref{eq_second_project_Z} and \autoref{eq_radicals_project}, the proof would be complete if we prove the inequality $\dim\big(R/\left((x:_R\mm_1^\infty)+\mm_1\right)\big) \ge (r_\fJ-1)+(p-1)=r+p-2$.
	
	By using  	\autoref{lem_radical_sats} and 
	\autoref{lem_dim_projects}(i) we obtain that 
	$$
	\dim\left(R/(x:_R\mm_1^\infty)\right) = \dim\left(R/(x:_R\nn^\infty)\right) = (r-1)+p = r+p-1,
	$$
	and since $r(\fJ)=r$, we have $$\dim(R/\mm_1) =\dim(R_{(\fJ)}) = r(\fJ)+(p-1)= r +(p-1)=r+p-1.$$ 
	  Moreover, \autoref{eq_radicals_project} and  \autoref{lem_dim_projects}(i)  yield that
$$
	\dim\big(R/\left((x:_R\mm_1^\infty)+\mm_1\right)\big) = \dim\big(R/\left((x:_R\nn^\infty)+\mm_1\right)\big) \le (r-1)+(p-1)=r+p-2.
	$$
	Since $x \in \mm_1$, \autoref{rem_radicals}(ii) gives that $\ara\left((x:_R\mm_1^\infty)\cap \mm_1\right) = \ara\left((x)\right)=1$.
	As $R$ is a domain, $c(R)=\sdim(R)=r+p$.
	Therefore, from \autoref{lem_dimension_connect}  we obtain that 
	$$
	\dim\big(R/\left((x:_R\mm_1^\infty)+\mm_1\right)\big) \ge \min\{ r+p, (r+p)-1 \} - 1 = r+p-2.
	$$
	So, the proof is complete.
\end{proof}


\begin{notation}
	\label{nota_generic}
	Let $\{x_0,\ldots,x_s\}$ be a basis of the $\kk$-vector space $[R]_{\ee_1}$.
	Consider a purely transcendental field extension $\LL := \kk(z_0,\ldots,z_s)$ of $\kk$, and set $R_\LL := R \otimes_\kk \LL$ and $X_\LL := X \otimes_\kk \LL = \multProj\left(R_\LL\right) \subseteq \PP \otimes_\kk \LL = \PP_\LL^{m_1} \times_\LL \cdots \times_\LL \PP_\LL^{m_p}$.
	We say that $z := z_0x_0 + \cdots + z_sx_s \in {\left[R_\LL\right]}_{\ee_1}$ is the \emph{generic element} of ${\left[R_\LL\right]}_{\ee_1}$.
\end{notation}

In the following remark we explain that field extensions as in \autoref{nota_generic} preserve the domain assumption.

\begin{remark}
	\label{rem_infinite_field}
	Suppose that $R$ is a domain and consider a purely transcendental field extension $\kk(\xi)$. Then, $R \otimes_\kk \kk(\xi)$ is also a domain; indeed, one can see that $R \otimes_\kk \kk(\xi)$ is a subring of the field of fractions $\Quot(R[\xi])$ of the polynomial ring $R[\xi]$.
	So, when $R$ is a domain one can extend $\kk$ to an infinite field without loosing the assumption of $R$ being a domain.    
\end{remark}

The lemma below shows that the Hilbert function modulo a generic element coincides with the one module a general element.

\begin{lemma}
	\label{lem_generic}
	Assume \autoref{nota_generic} with  $\kk$ being an infinite field.   Let $x \in {[R]}_{\ee_1}$  be  a general element, then  
		$$
		\dim_\kk\big(\left[R/xR\right]_\nu\big) = \dim_\LL\big(\left[R_\LL/zR_\LL\right]_\nu\big)
		$$
		for all $\nu \in \NN^p$.
\end{lemma}
\begin{proof}
	Let $T$ be the polynomial ring $T = \kk[z_0,\ldots,z_s]$ and consider the finitely generated $T$-algebra given by $S = \left(R \otimes_\kk T\right)/w\left(R \otimes_\kk T\right)$ where $w = z_0x_0+\cdots+z_sx_s \in R \otimes_\kk T$.
	From the Grothendieck's Generic Freeness Lemma (see, e.g., \cite[Theorem 24.1]{MATSUMURA}, \cite[Theorem 14.4]{EISEN_COMM}) there exists an element $0 \neq a \in T$ such that $S_a$ is a free $T_a$-module.
	Hence, for any $\pp \in \Spec(T)$ inside the dense open subset $D(a) \subset \Spec(T)$, if $k(\pp)$ denotes the residue field $k(\pp) = T_\pp/\pp T_\pp$ of $T_\pp$, one has that 
	$$
	\dim_{k(\pp)}\big(\left[S_a\otimes_{T_a} k(\pp)\right]_\nu\big) = \dim_{\Quot(T)}\big(\left[S_a\otimes_{T_a} \Quot(T)\right]_\nu\big) = \dim_\LL\big(\left[R_\LL/zR_\LL\right]_\nu\big)
	$$
	for all $\nu \in \NN^p$. 
	Note that for any $\beta = (\beta_0,\ldots,\beta_s) \in \kk^{s+1}$ with $\pp_\beta=(z_0-\beta_0,\ldots,z_s-\beta_s) \in D(a)$ one has the isomorphisms 
	$$
	S_a\otimes_{T_a} k(\pp_\beta) \;\cong\; \frac{R\otimes_\kk T}{\left(z_0x_0+\cdots+z_sx_s, z_0-\beta_0,\ldots, z_s-\beta_s\right)} \;\cong\; R / \left(\beta_0x_0+\cdots+\beta_sx_s\right)R.
	$$
	So, the result follows.
\end{proof}

We now obtain \autoref{thmA} when $X$ is an irreducible scheme.

\begin{remark}
	We first provide a couple of general words regarding the proof of \autoref{thm_main_irreducible} below and where the irreducibility assumption comes into play.
	The proof is achieved by iteratively cutting with generic hyperplanes (following \autoref{nota_generic}) to arrive to a zero-dimensional situation, and the main constraint is to control the dimension of the image of all the possible projections after cutting with a general hyperplane (see \autoref{eq_induction_step_main}).
	Our main tool to control those dimensions is \autoref{thm_projections}, where it is needed to assume that $R$ is a domain. 
	When $X$ is irreducible, by just taking the reduced scheme structure $X_\text{red} = \multProj(R/\sqrt{0})$ we can easily reduce to the case where $R$ is a domain.
	To maintain the irreducibility assumption during the inductive process, we use a ``generic'' version of Bertini's Theorem as presented in \cite[Proposition 1.5.10]{FLENNER_O_CARROLL_VOGEL}.
	It should be noted that the usual versions of Bertini's Theorem for irreducibility require $X$ to be geometrically irreducible and that the dimension of the image of certain morphism is bigger or equal than two (see \cite[Theoreme 6.10, Corollaire 6.11]{JOUANOLOU_BERTINI}).
	Finally, \autoref{lem_generic} is used to relate the process of cutting with a \emph{generic} hyperplane with the one of cutting with a \emph{general} hyperplane.
\end{remark}

\begin{theorem}
	\label{thm_main_irreducible}
	Assume \autoref{setup_initial}.
	Suppose that $X$ is irreducible.
	Let $\bn=(n_1,\ldots,n_p) \in \NN^p$ such that $\lvert \bn \rvert=r$.
	Then, $\deg_\PP^\bn(X) > 0$ if and only if for each $\fJ = \{j_1,\ldots,j_k\} \subseteq [p]$ the inequality 
	$
	n_{j_1} + \cdots + n_{j_k} \le r(\fJ)
	$
	holds.
\end{theorem}
\begin{proof}
	From \autoref{thm_deg_Hilb_pol} it is clear that the inequalities $n_{j_1} + \cdots + n_{j_k} \le r(\fJ)$ are a necessary condition for $\deg_\PP^\bn(X)=e(\bn;R) >0$.
	Therefore, it suffices to show that they are also sufficient.
	
	Assume that $
	n_{j_1} + \cdots + n_{j_k} \le r(\fJ)
	$ 	for every $\fJ = \{j_1,\ldots,j_k\} \subseteq [p]$. We  may also assume that $(0:_R\nn^\infty)=0$  by  \autoref{modSat}(i). 
	Hence, the condition of $X$ being irreducible implies that $\sqrt{0} \subset R$ is a prime ideal. 
	Since the associativity formula for mixed multiplicities (see, e.g., \cite[Lemma 2.7]{MIXED_MULT}) yields that 
	$$
	e(\bn;R)=\text{length}_{R_{\sqrt{0}}}\left(R_{\sqrt{0}}\right)\cdot e\left(\bn;R/{\sqrt{0}}\right),
	$$ 
	we can assume that $R$ is a domain, and we do so.
	In addition, by \autoref{modSat}(ii),  \autoref{thm_deg_Hilb_pol}, and \autoref{rem_infinite_field}	
	we may  also assume that $\kk$ is an infinite field.
	
	We proceed by induction on $r$. If $r = 0$, then \cite[Theorem 3.10]{MIXED_MULT} implies  $e(\mathbf{0};R)>0$.
	
	Suppose now that $r \ge 1$.
	Without any loss of generality, perhaps after changing the grading, we can assume that $n_1 \ge 1$.
	Let $\LL$, $R_\LL$, $X_\LL$ and $z$ be defined as in \autoref{nota_generic}.
	Let $x \in [R]_{\ee_1}$ be a general element. 
	Set $S = R/xR$, $Z = X \cap V(x) = \multProj(S)$, $T = R_\LL/zR_\LL$, $W = X_\LL \cap V(z) = \multProj(T)$ and $\bn'=\bn-\ee_1$.
	Then, \cite[Lemma 3.9]{MIXED_MULT} and \autoref{lem_generic} yield that $e(\bn;R)=e(\bn';S)=e(\bn';T)$.	
	From  \cite[Proposition 1.5.10]{FLENNER_O_CARROLL_VOGEL} we obtain that $W$ is also an irreducible scheme.
	By the assumed inequalities and because  $n_1\ge 1$ we have that for each $\fJ = \{j_1,\ldots,j_k\} \subseteq [p]$ the following inequality holds
	\begin{align}
	\label{eq_induction_step_main}
	n_{j_1}' + \cdots + n_{j_k}' \le  
	\min\Big\{ r(\fJ),\; r\big(\fJ \cup \{1\}\big)-1\Big\},
	\end{align} 
and the latter is equal to $\dim(\Pi_\fJ(Z))$ by \autoref{thm_projections}.  Moreover, by 
 \autoref{lem_generic} and \autoref{thm_deg_Hilb_pol}, we also have $ \dim(\Pi_\fJ(W))=\dim(\Pi_\fJ(Z))$; here, by an abuse of notation $\Pi_\fJ(W)$ denotes the image of the natural projection 
	$
	\Pi_\fJ: \PP \otimes_\kk \LL \rightarrow \PP_\LL^{m_{j_1}} \times_\LL \cdots \times_\LL \PP_\LL^{m_{j_k}}
	$	
	restricted to $W$.
	
	Finally, by using the inductive hypothesis applied to the irreducible scheme $W$, we obtain that $e(\bn;R)=e(\bn';T) > 0$, and so the result follows.
\end{proof}

Now we are ready to  show the general version of \autoref{thmA}.

\begin{corollary}
	\label{thm_main}
	Assume \autoref{setup_initial}.
	Let $\bn=(n_1,\ldots,n_p) \in \NN^p$ such that $\lvert \bn \rvert=\dim(X)$.
	Then, $\deg_\PP^\bn(X) > 0$ if and only if there is an irreducible component $Y \subseteq X$ of $X$ that satisfies the following two conditions:
	\begin{enumerate}[\rm (a)]
		\item $\dim(Y) = \dim(X)$.
		\item For each $\fJ = \{j_1,\ldots,j_k\} \subseteq [p]$ the inequality 
		$
		n_{j_1} + \cdots + n_{j_k} \le \dim\big(\Pi_\fJ(Y)\big)
		$
		holds.
	\end{enumerate} 
\end{corollary}
\begin{proof}
We may assume that $(0:_R\nn^\infty)=0$ by  \autoref{modSat}(i).
	By the associativity formula for mixed multiplicities (see, e.g., \cite[Lemma 2.7]{MIXED_MULT}) we get the equation
	$$
	\deg_\PP^\bn(X) \;=\; e(\bn;R) \;=\; \sum_{\substack{\fP \in \Min(R)\\ \dim(R/\fP) = r+p}} \text{length}_{R_\fP}(R_\fP)\cdot
	e(\bn; R/\fP).
	$$
	Thus, $e(\bn;R) > 0$ if and only if $e(\bn;R/\fP) > 0$ for some minimal prime $\fP \in \Min(R)$ of maximal dimension. 
	So, the result is clear from \autoref{thm_main_irreducible}.
\end{proof}

Below we have a proof for \autoref{corB}.

\begin{corollary}
	\label{cor_positive_mixed_mult}
	Assume \autoref{setup_Artinian}.
	Let $\bn=(n_1,\ldots,n_p) \in \NN^p$ such that $\dim\left(R/\left(0:_R\nn^\infty\right)\right) - p= \lvert \bn \rvert$.
	Then, $e(\bn;R) > 0$ if and only if there is a minimal prime ideal $\fP \in \Min\left(0:_R\nn^\infty\right)$ of $\left(0:_R\nn^\infty\right)$ that satisfies the following two conditions:
	\begin{enumerate}[\rm (a)]
		\item $\dim\left(R/\fP\right)  = \dim\left(R/\left(0:_R\nn^\infty\right)\right)$.
		\item For each $\fJ = \{j_1,\ldots,j_k\} \subseteq [p]$ the inequality 
		$
		n_{j_1} + \cdots + n_{j_k} \le \dim\left(\frac{R}{\fP+\sum_{j \not\in \fJ} \mm_j}\right) - k
		$
		holds. 
	\end{enumerate} 
\end{corollary}
\begin{proof}
	As in \autoref{thm_main}, after assuming that $(0:_R\nn^\infty)=0$ and using the associativity formula for mixed multiplicities, we obtain that $e(\bn;R) > 0$ if and only if $e(\bn;R/\fP) > 0$ for some minimal prime $\fP \in \Min(R)$ of maximal dimension.  
	Note that, for each $\fP \in \Min(R)$,  $R/\fP$ is naturally a finitely generated standard $\NN^p$-graded algebra over a field.
	So, the result follows by using \autoref{thm_main_irreducible}.
\end{proof}

Finally, for the sake of completeness, we provide a brief discussion on how \autoref{thm_main_irreducible} can be recovered (over the complex number) from the related results of \cite[\S 2.2]{KAVEH_KHOVANSKII}.

\begin{remark}
	\label{rem_Kaveh_Kho}
	Assume $\kk = \CC$.
	For the closed subscheme $X \subset \PP = \PP_\kk^{m_1} \times_\kk \cdots \times_\kk \PP_\kk^{m_p}$, let $L_i$ be the pullback of $\mO_{\PP_\kk^{m_i}}(1)$ to $X$. Take $E_i$ to be $|L_i|$. 
	Following the notation in \cite[\S 2.2]{KAVEH_KHOVANSKII}, for each $\emptyset\neq \fJ\subseteq [p]$, denote by 
	$$
	\Phi_\fJ \,:\, X \rightarrow \PP\left(E_\fJ^{\vee}\right)
	$$ the Kodaira map corresponding with the linear system $E_\fJ$.
	Let $\tau_\fJ$ be the dimension of the closure of the image of $\Phi_\fJ$ (\cite[Definition 2.12]{KAVEH_KHOVANSKII}).
	Consequently, it is easy to check that $\dim\left(\Pi_{\fJ}(X)\right)=\tau_{\fJ}$.
	Thus, \cite[Theorems 2.14, 2.19]{KAVEH_KHOVANSKII} translate into the following statement: $\dim\left(\Pi_{\fJ}(X)\right) \ge |\fJ|$ if and only if for general hyperplanes $H_j\in |\Pi_j^*\mO_{\PP_\kk^{m_i}}(1)|$ ($j\in\fJ$), $X\cap(\bigcap_{j\in\fJ}H_j)\neq\emptyset$. The latter is equivalent to the condition $[X]\cdot \prod_{j\in\fJ} [H_j]\neq 0$ on intersection of classes. \autoref{thm_main_irreducible} (over the complex numbers) eventually follows from applying this statement finitely many times to relevant index subsets $\fJ$.
\end{remark}


\section{Positivity of the mixed multiplicities of ideals}\label{sec_mix_ideals}

In this section, we characterize the positivity of the mixed multiplicities of ideals.
The results obtained here are a consequence of applying \autoref{corB} to a certain multigraded algebra.
For the particular case of ideals generated in one degree in graded domains we obtain a neat characterization in  \autoref{thm_equigen_ideals}.

Throughout this section we use the following setup. 

\begin{setup}
	\label{setup_mixed_mult_ideals}
	Let $R$ be a Noetherian local ring with maximal ideal $\mm \subset R$ (or a finitely generated standard graded algebra over a field $\kk$ with graded irrelevant ideal $\mm \subset R$).
\end{setup}

Let $J_0 \subset R$ be an $\mm$-primary ideal and $J_1,\ldots,J_p \subset R$ be arbitrary ideals (homogeneous in the graded case).
The \textit{multi-Rees algebra} of the ideals $J_0,J_1,\ldots,J_p$ is given by 
$$
\Rees(J_0,\ldots,J_p) \;:=\; R[J_0t_0,\ldots,J_pt_p] \;=\; \bigoplus_{i_0 \ge 0, \ldots,i_p \ge 0} J_0^{i_0}\cdots J_p^{i_p} t_0^{i_0}\cdots t_p^{i_p} \;\subset\; R[t_0,\ldots,t_p],
$$
where $t_0,\ldots,t_p$ are new variables.
Note that $\Rees(J_0,\ldots,J_p)$ is naturally a standard $\NN^{p+1}$-graded algebra and that, for $0 \le k \le p$, the ideal $\mm_k$ generated by elements of degree $\ee_k$ is given by 
$$
\mm_k \;:=\; J_kt_k \, \Rees(J_0,\ldots,J_p) \; \subset \; \Rees(J_0,\ldots,J_p).  
$$
Let $\nn:=\mm_0 \cap \cdots \cap \mm_p$ be the corresponding multigraded irrelevant ideal.
Since $J_0$ is $\mm$-primary, we obtain that
$$
T(J_0\mid J_1,\ldots,J_p) \;:=\; \Rees(J_0,\ldots,J_p) \otimes_R R/J_0 \;=\; \bigoplus_{i_0 \ge 0, i_1 \ge 0 \ldots,i_p \ge 0} J_0^{i_0}J_1^{i_1}\cdots J_p^{i_p} \big/ J_0^{i_0+1}J_1^{i_1}\cdots J_p^{i_p}
$$
is a finitely generated standard $\NN^{p+1}$-graded algebra over the Artinian local ring $R/J_0$.
For simplicity of notation, throughout this section we fix
$\Rees := \Rees(J_0,\ldots,J_p)$  
and
$T  := T(J_0\mid J_1,\ldots,J_p)$.
Let $r$ be the integer 
$
r := \dim\left(\multProj\left(T\right)\right),
$
which coincides with the degree of the Hilbert polynomial $P_{_{T} }(u_1,\ldots,u_{p+1})$ of the $\NN^{p+1}$-graded $R/J_0$-algebra $T$.
From \cite[Theorem 1.2(a)]{TRUNG_VERMA_MIXED_VOL} we have the equality $r = \dim\left(R/(0:_R(J_1\cdots J_p)^\infty)\right)-1$.

\begin{definition}
	Under the above notations, for each $\bn \in \NN^{p+1}$ with $\lvert \bn \rvert = r$, we say that
	$$
	e_\bn\left(J_0\mid J_1,\ldots,J_p\right) \;:= \; e\left(\bn; T\right)
	$$
	is the \emph{mixed multiplicity of $J_0,J_1,\ldots,J_p$ of type $\bn$}.
\end{definition}

The main focus in this section is to characterize when $e_\bn\left(J_0\mid J_1,\ldots,J_p\right) > 0$.
As a direct consequence of \autoref{corB} we get the following general criterion for the positivity $e_\bn\left(J_0\mid J_1,\ldots,J_p\right)$.

\begin{corollary}
	\label{cor_mixed_mult_ideals}
	Assume \autoref{setup_mixed_mult_ideals} and the notations above.
	Let $\bn=(n_0,n_1,\ldots,n_p) \in \NN^{p+1}$ such that $\lvert \bn \rvert = r$.
	Then, $e_\bn\left(J_0\mid J_1,\ldots,J_p\right) > 0$ if and only if there is a minimal prime ideal $\fP \in \Min\left(0:_T\nn^\infty\right)$ of $\left(0:_T\nn^\infty\right)$ that satisfies the following two conditions:
	\begin{enumerate}[\rm (a)]
		\item $\dim\left(T/\fP\right)  = \dim\left(T/\left(0:_T\nn^\infty\right)\right)$.
		\item For each $\fJ = \{j_1,\ldots,j_k\} \subseteq \{0\} \cup [p]$ the inequality,
		$
		n_{j_1} + \cdots + n_{j_k} \le \dim\left(\frac{T}{\fP+\sum_{j \not\in \fJ} \mm_jT}\right) - k
		$
		holds.
	\end{enumerate} 
\end{corollary}

We now focus on the case where $R$ is a graded $\kk$-domain and each ideal $J_i$ is generated in one degree.
In this case, our characterization depends on the \emph{analytic spread} of certain ideals; recall that the analytic spread of an ideal $I \subset R$ is given by $\ell(I) := \dim\Big(\Rees(I)/\mm\Rees(I)\Big)$.

\begin{theorem}
	\label{thm_equigen_ideals}
	Let $R$ be a finitely generated standard graded domain over a field $\kk$ with graded irrelevant ideal $\mm \subset R$.
	Let $J_0 \subset R$ be an $\mm$-primary ideal and $J_1,\ldots,J_p \subset R$ be arbitrary ideals.
	Suppose that, for each $0 \le i \le p$, $J_i$ is generated by homogeneous elements of the same degree $d_i>0$.
	Let $\bn=(n_0,n_1,\ldots,n_p) \in \NN^{p+1}$ such that $\lvert \bn \rvert = \dim(R)-1$.
	Then, $e_\bn\left(J_0\mid J_1,\ldots,J_p\right) > 0$ if and only if for each $\fJ = \{j_1,\ldots,j_k\} \subseteq [p]$ the inequality 
	$$
	n_{j_1} + \cdots + n_{j_k} \;\le\; \ell\big(J_{j_1}\cdots J_{j_k}\big) - 1
	$$
	holds. 
\end{theorem}
\begin{proof}
	First, note that $r = \dim(R)-1$.
	
	Since $J_0$ is $\mm$-primary, the kernel of the canonical map $T \twoheadrightarrow T' := \Rees \otimes_R R/\mm$ is nilpotent. 
	Therefore, the conditions (a),(b) in \autoref{cor_mixed_mult_ideals} are satisfied for $T$ if and only if they are satisfied for $T'$.
	
	Consider the $\NN^{p+1}$-graded domain given by 
	$$
	\FF \;:=\; \bigoplus_{i_0 \ge 0, \ldots,i_p \ge 0} \left[J_0^{i_0}\right]_{i_0d_0}\left[J_1^{i_1}\right]_{i_1d_1}\cdots \left[J_p^{i_p}\right]_{i_pd_p} t_0^{i_0}t_1^{i_1}\cdots t_p^{i_p} \; \subset \; \Rees.
	$$	
	Since ${[J_0^{i_0}]}_{i_0d_0}{[J_1^{i_1}]}_{i_1d_1}\cdots {[J_p^{i_p}]}_{i_pd_p} \cong J_0^{i_0}J_1^{i_1}\cdots J_p^{i_p} \otimes_R R/\mm$, we have the isomorphism $T' \cong \FF$ and so $T'$ is a domain.
	
	For any $\fK = \{h_1,\ldots,h_s\} \subseteq \{0\} \cup [p]$, since we have the natural isomorphism 
	$
	\Rees\big/\sum_{h \not\in \fK}\mm_h \;\cong \; \Rees(J_{h_1},\ldots,J_{h_s}) = R[J_{h_1}t_{h_1},\ldots,J_{h_s}t_{h_s}],
	$
	it follows that 
	$$
	\dim\big(T'\big/\sum_{h \not\in \fK}\mm_hT'\big) = \dim\big(\Rees(J_{h_1},\ldots,J_{h_s}) \otimes_R R/\mm\big).
	$$
	After using the Segre embedding we get the isomorphism $\multProj\big(\Rees(J_{h_1},\ldots,J_{h_s}) \otimes_R R/\mm\big) \cong \Proj\big(\Rees(J_{h_1}\cdots J_{h_s}) \otimes_R R/\mm\big)$ and,  accordingly, from \autoref{lem_basics_projections}(i)  we have 	
	$$
	\dim\big(\Rees(J_{h_1},\ldots,J_{h_s}) \otimes_R R/\mm\big) = \dim\big(\Rees(J_{h_1}\cdots J_{h_s}) \otimes_R R/\mm\big) +s-1 = \ell(J_{h_1}\cdots J_{h_s})+s-1,
	$$
(also, see \cite[Corollary 3.10]{bivia2020analytic}). 

So, $e_\bn\left(J_0\mid J_1,\ldots,J_p\right) > 0$ if and only if for each $\fK = \{h_1,\ldots,h_s\} \subseteq {0} \cup [p]$ the inequality 
	$
	n_{h_1} + \cdots + n_{h_s} \;\le\; \dim(T'/\sum_{h \not\in \fK}\mm_hT') - s = \ell(J_{h_1}\cdots J_{h_s})-1
	$
	holds. 
		
	For any $\fK = \{0,h_2,\ldots,h_s\} \subseteq \{0\} \cup [p]$, as $J_0$ is $\mm$-primary, from \cite[Theorem 5.1.4, Proposition 5.1.6]{huneke2006integral} we obtain
	\begin{align*}
	\dim\big(T'\big/\sum_{h \not\in \fK}\mm_hT'\big) =  \dim\big(\Rees(J_0,J_{h_2},\ldots,J_{h_s}) \otimes_R R/J_0\big) &= \dim\Big(\gr_{_{J_0\Rees(J_{h_2},\ldots,J_{h_s})}}\big(\Rees(J_{h_2},\ldots,J_{h_s})\big) \Big)\\
	&=\dim(R) + s -1.
	\end{align*}
	Therefore, we only need check the inequalities corresponding to the subsets $\fJ \subseteq [p]$, and so the result follows.
\end{proof}

\begin{remark}
	Note that in \autoref{thm_equigen_ideals} the conditions for the positivity of $e_\bn\left(J_0\mid J_1,\ldots,J_p\right)$ do not involve the $\mm$-primary ideal $J_0$ (see  \cite[Corollary 1.8(a)]{TRUNG_VERMA_MIXED_VOL}).
\end{remark}

\begin{remark}
We note that if in \autoref{thm_equigen_ideals} we have $\ell(J_i)=\dim(R)$ for every $1\le i\le p$, then by \cite[Lemma 4.7]{hyry2002cohen} for each $\fJ = \{j_1,\ldots,j_k\} \subseteq [p]$ we also have $ \ell\big(J_{j_1}\cdots J_{j_k}\big)=\dim(R)$. Therefore, by \autoref{thm_equigen_ideals} it follows that $e_\bn\left(J_0\mid J_1,\ldots,J_p\right) > 0$ for every $\bn \in \NN^{p+1}$ such that $\lvert \bn \rvert = \dim(R)-1$.
\end{remark}



\section{Polymatroids}\label{sec_polym}

We recall that \autoref{thmA} implies that $\msupp_\PP(X)$ (see \autoref{defMsupp}) is the set of integer points in a polytope when $X$ is irreducible.  In this section we explore properties of these discrete sets.

Following standard notations, we say that $X$ is a \emph{variety} over $\kk$ if $X$ is a reduced and irreducible separated scheme of finite type over $\kk$ (see, e.g., \cite[\href{https://stacks.math.columbia.edu/tag/020C}{Tag 020C}]{stacks-project}). In the following  two results we connect the theory of polymatroids (see \autoref{sub_Poly}) with $\msupp_\PP(X)$ when $X$ is a variety.

\begin{proposition}\label{prop_poly}
	Let $X\subseteq \PP=\PP_{\kk}^{m_1} \times_\kk \cdots \times_\kk \PP_{\kk}^{m_p}$ be a multiprojective variety over an arbitrary field $\kk$. 
	Then $\msupp_\PP(X)$ is a discrete algebraic polymatroid over $\kk$.
\end{proposition}

\begin{proof}
	We consider the associated $\NN^p$-graded $\kk$-domain $R$.	
	Let $\xi$ be the generic point of $X$ and set $\LL:=\OO_{X, \xi}$.
	For each $\fJ \subseteq [p]$, let $X_\fJ  = \Pi_\fJ(X) = \multProj\left(R_{(\fJ)}\right)$ and $\xi_\fJ$ be the generic point of  $X_\fJ$, and notice that 
	$$
	\OO_{X_\fJ, \xi_\fJ} = \big\{f/g \mid f,g \in R_{(\fJ)}, \;g \neq 0, \; \deg(f)=\deg(g) \big\} \;\subseteq\; \LL.
	$$
	For $1\le i \le p$, let $\LL_i := \OO_{X_i, \xi_i}$.
	Then, for each $\fJ \subseteq [p]$, we have that $\OO_{X_\fJ, \xi_\fJ} = \bigwedge_{j\in \fJ} \LL_j$
	and that $\dim\left(X_\fJ\right) = \trdeg_\kk\left(\OO_{X_\fJ, \xi_\fJ}\right)$ (see, e.g., \cite[Theorem 5.22]{GORTZ_WEDHORN}, \cite[Exercise II.3.20]{HARTSHORNE}).
	Finally, the result follows from \autoref{thmA}.			
\end{proof}


In  \cite[Corollary 2.8]{TRUNG_POSITIVE} it is shown that the conclusion of   \autoref{thm_main_irreducible} holds if $p=2$ and $X$ is arithmetically Cohen-Macaulay. The following example shows that this result does not always hold for  $p> 2$. 

\begin{example}\label{ex_cm}
Let $S=\kk[x_1,\ldots,x_{12},y_1,\ldots,y_{12}]$  be a polynomial rings with an $\NN^{12}$-grading induced by  $\deg(x_i)=\deg(y_i)=\ee_i$  for $1\le i\le 12$. 
Let $\Delta$ be the simplicial complex given by the boundary of the icosahedron. 
We note that $\Delta$ is a Cohen-Macaulay complex (because it is a triangulation of the sphere $\mathbb{S}^2$ \cite[Corollary II.4.4]{STANLEY}), but it is not a (poly)matroid (see \autoref{rem:matroid}) since not every restriction is pure \cite[Proposition III.3.1]{STANLEY}.

Let $J_\Delta=\{(x_{i_1}+y_{i_1})\cdots (x_{i_k}+y_{i_k})\mid \{i_1,\ldots,i_k\}\notin\Delta\}$ and  $X_\Delta=\multProj(S/J_{\Delta})\subset \PP=\PP_\kk^1\times\cdots\times \PP_\kk^1$. The definition of $J_\Delta$ is a modification on the definition of  $I_\Delta$, the Stanley-Reisner ideal of $\Delta$ with monomials in the variables $\{x_1,\ldots,x_{12}\}$ \cite[Chapter 1]{MILLER_STURMFELS}, \cite[Chapter II]{STANLEY}. 
It can be easily verified that $I_\Delta$ is the  initial ideal of $J_\Delta$ with respect to any elimination order with  $\{x_1,\ldots,x_{12}\} \ge \{y_1,\ldots,y_{12}\}$. Since the ideal $J_\Delta$ is obtained  from $I_\Delta$ by a linear change of variables, we have a similar primary decomposition as \cite[Theorem 1.7]{MILLER_STURMFELS}, so no component is supported on any coordinate subspace and thus $J_\Delta$ is saturated with respect to the irrelevant ideal of $S$. 
By \cite[Corollary 3.3.5]{HERZOG_HIBI_MONOMIALS}, $X_\Delta$ is arithmetically Cohen-Macaulay. 
 Moreover, since Hilbert functions are preserved by Gr\"obner degenerations, the multidegree of $X_{\Delta}$ coincides with $\mathcal{C}(S/I_\Delta;\ttt)$ (see \autoref{rem_Hilb_series}). Thus,  $\msupp_\mathbb{P}(X_\Delta)$ consists of all the incidence vectors of the facets of $\Delta$ \cite[Theorem 1.7]{MILLER_STURMFELS} and then it is not a polymatroid.
\end{example}

With \autoref{prop_poly} in hand, we can introduce the following class of polymatroids.

\begin{definition}\label{def_chow}
A  polymatroid $\mathcal{P}$ is \textit{Chow} over a field $\kk$ if there exists a variety $X\subseteq \PP=\PP_{\kk}^{m_1} \times_\kk \cdots \times_\kk \PP_{\kk}^{m_p}$ such that $\msupp_\PP(X)=\mathcal{P}$.
\end{definition}
The following statement follows as an easy corollary of the main result in \cite{BINGLIN}, when $\kk$ is an infinite field. Here we give a simple direct argument for an arbitrary field $\kk$.
\begin{proposition}\label{prop_linear}
A linear polymatroid over an arbitrary field $\kk$ is Chow over the same field.
\end{proposition}
\begin{proof}
Let $V$ be a $\kk$-vector space and $V_1,\ldots,V_p$ be arbitrary subspaces. 
Let $S$ be the polynomial ring $S = \Sym(V) = \kk[x_1,\ldots,x_q]$, where $q = \dim_\kk(V)$.
By using the isomorphism $[S]_1 \cong V$, we identify each $V_i$ with a $\kk$-subspace $U_i$ of $[S]_1$.
For each $1 \le i \le p$, let $\{x_{i,1},x_{i,2},\ldots,x_{i,r_i}\} \subset [S]_1$ be a basis of the $\kk$-vector space $U_i$.
Let $T$ be the $\NN^p$-graded polynomial ring 
$$
T := \kk\left[y_{i,j} \mid 1 \le i \le p, \,0 \le j \le r_i,\, \deg(y_{i,j}) = \ee_i\right].
$$
Induce an $\NN^p$-grading on $S[t_1,\ldots,t_p]$ given by $\deg(t_i) = \ee_i$ and $\deg(x_j) = \mathbf{0}$.
Consider the $\NN^p$-graded $\kk$-algebra homomorphism 
$$
\varphi = T \rightarrow S[t_1,\ldots,t_p], \qquad 
\begin{array}{l}
	y_{i,0} \mapsto  t_i \quad\;\, \text{ for } \;  1 \le i \le p \\
	y_{i,j} \mapsto  x_{i,j}t_i \; \text{ for } \;  1 \le i \le p,\, 1 \le j \le r_i.
\end{array}
$$
Note that $\fP := \Ker(\varphi)$ is an $\NN^p$-graded prime ideal. 
Set $R:= T/\fP$ and $X := \multProj(R)$.
By construction, for each $\fJ \subseteq [p]$, we obtain the isomorphism 
$$
R_{(\fJ)} \;\cong\; \kk\left[x_{i,j}t_i \mid  i \in \fJ,\, 1 \le j \le r_i\right]\left[t_i\mid i \in \fJ\right] \;\subset\; S[t_1,\ldots,t_p];
$$
thus, it is clear that 
\begin{align*}
	\dim\left(R_{(\fJ)}\right) &= \trdeg_\kk\big(\kk\left[x_{i,j}t_i \mid  i\in \fJ,\, 1 \le j \le r_i\right]\left[t_i\mid i \in \fJ\right]\big)\\
	&= \trdeg_\kk\big(\kk\left[x_{i,j}\mid  i\in \fJ,\, 1 \le j \le r_i\right]\left[t_i\mid i \in \fJ\right]\big)\\
	&= \dim_\kk\left(\sum_{i\in\fJ} U_i\right) + \lvert \fJ \rvert = \dim_\kk\left(\sum_{i\in\fJ} V_i\right) + \lvert \fJ \rvert.
\end{align*}
Therefore, \autoref{lem_basics_projections} yields that $\dim\left(\Pi_{\fJ}(X)\right)= \dim_\kk\left(\sum_{i\in\fJ} V_i\right)$, and so the result follows from \autoref{thmA}.
\end{proof}

%

The following is the main theorem of this section. Here we summarize the results presented above to show that the class of Chow polymatroids lies in between the ones introduced in \autoref{def_algebraic_polymatroid}.

\begin{theorem}\label{thm_classification}
Over an arbitrary field $\kk$, we have the following inclusions of families of polymatroids
$$
\Big(\texttt{Linear polymatroids}\Big) \;\subseteq\; \Big(\texttt{Chow polymatroids}\Big) \;\subseteq\; \Big(\texttt{Algebraic polymatroids}\Big).
$$	
Moreover, when $\kk$ is a field of characteristic zero, the three families coincide.
\end{theorem}
\begin{proof}
The first inclusion follows from \autoref{prop_linear}; the second from \autoref{prop_poly}.
In the characteristic zero case linear and algebraic polymatroids coincide by \cite[Corollary, Page 166]{INGLETON} (also, see \autoref{IngletonRem}). 
\end{proof}
\begin{remark}\label{IngletonRem}
The result mentioned above from \cite{INGLETON} is stated for matroids but the arguments go unchanged for polymatroids.
\end{remark}
\begin{remark}\label{rem_Alg_not_line}
Over finite fields there are algebraic matroids that are not linear. An example is the Non-Pappus matroid described in \cite[Page 517]{OXLEY}, it is algebraic over any field of positive characteristic but not linear over any field.
\end{remark}
Classifying linear polymatroid rank functions is a difficult problem.  For linear matroids over a field of characteristic zero, the poetically titled \textit{``The missing axiom of matroid theory is lost forever''} \cite{AXIOM_LOST} together with a recent addition \cite{AXIOM_LOST_RECENT} shows that there is no finite list of axioms that characterize which rank functions are linear. For fields of positive characteristic,  Rota conjectured in 1971 that for each field there is a list of finite restrictions. A proof of Rota's conjecture has been announced by Geelen, Gerards, and Whittle, but expected to be several hundred of pages long. Little is known about the algebraic case. In \cite{NONALGEBRAIC} there is an example of a matroid that is not algebraic over any field: the Vamos matroid $V_8$ \cite[Page 511]{OXLEY}. For these reasons we do not expect a further characterizations of Chow polymatroids.

We finish this section with the following question.

\begin{question}
Are all algebraic polymatroids Chow? 
\end{question}

\section{Applications}\label{sec_comb} 

In this section  we  relate our results to several objects from combinatorial algebraic geometry.

\subsection{Schubert polynomials}
\label{subsect_Schubert}

Let $\fS_p$ be the symmetric group on the set $[p]$. For every $i\in[p-1]$ we have the transposition $s_i:=(i,i+1)\in \fS_p$. Recall that the set $S=\langle s_i, 1\leq i< p \rangle$ generates $\fS_p$. The \emph{length} $l(\pi)$ of a permutation $\pi$ is the least amount of elements in $S$ needed to obtain $\pi$. Alternatively, the length is equal to the number of \emph{inversions}, i.e., $\l(\pi)=\{(i,j)\in[p]\times[p]:i<j,\; \pi(i)>\pi(j)\}$. 
The permutation $\pi_0=(p,p-1,\cdots,2,1)$ (in one line notation) is the longest permutation, it has length $\frac{p(p-1)}{2}$.
\begin{definition}
The {\it Schubert polynomials} $\fS_\pi\in\ZZ[t_1,\ldots,t_p]$ are defined recursively  in the following  way. 
First we define $\fS_{\pi_0}:=\prod_{i} t_i^{p-i}$, and for any permutation $\pi$ and transposition $s_i$ with $l(s_i\pi)<l(\pi)$ we let
\[
\fS_{s_i\pi} = \dfrac{\fS_{\pi}-s_i\fS_{\pi}}{t_i-t_{i+1}},
\]
where $\fS_{p}$ acts on $\ZZ[t_1,\ldots,t_p]$ by permutation of variables. For more information see \cite[Chapter 10]{FULTON_TABLEAUX}. 
\end{definition}

Next we define {\it matrix Schubert varieties} following \cite[Chapter 15]{MILLER_STURMFELS}. Let $\kk$ be an algebraic closed field and  $M_p(\kk)$ be the $\kk$-vector space of $p\times p$ matrices with entries in $\kk$. As an affine variety we define its coordinate ring as $R_p:=\kk[x_{ij}:(i,j)\in[p]\times[p]]$. Furthermore we consider an $\NN^p$-grading on $R_p$ by letting $\deg(x_{ij}) = \ee_i$.

\begin{definition} 
Let $\pi$ be a permutation matrix. The matrix Schubert variety $\overline{X_\pi}\subset M_p(\kk)$ is the subvariety
\[
\overline{X_\pi}=\{Z\in M_p(\kk)\mid \rank(Z_{m\times n})\leq \rank(\pi_{m\times n}) \; \text{ for all }\; m,n\},
\]
where $Z_{m\times n}$ is the restriction to the first $m$ rows and $n$ columns. This is an irreducible variety
 and the prime ideal $I\left(\overline{X_\pi}\right)$ is multihomogeneous \cite[Theorem 15.31]{MILLER_STURMFELS}. By  \cite[Theorem 15.40]{MILLER_STURMFELS}, the Schubert polynomial $\fS_\pi$  equals the multidegree polynomial of the variety corresponding to the ideal $I\left(\overline{X_\pi}\right)$ (see \autoref{defMsupp}).
\end{definition}

Following \cite{YONG} we say a polynomial $f=\sum_{\bn}c_{\bn}\ttt^\bn\in \ZZ[t_1,\cdots,t_p]$ have the Saturated Newton Polytope property (SNP for short) if $\text{supp}(f):=\{\bn\in\NN^p\mid c_\bn>0\}=\text{ConvexHull}\{\bn\in\NN^p\mid c_\bn>0\}\cap\NN^p$, in other words, if the support of $f$ consist of the integer points of a polytope.
In \cite[Conjecture 5.5]{YONG} it was conjectured that the Schubert polynomials have SNP property and they even conjectured a set of defining inequalities for the Newton polytope in \cite[Conjecture 5.13]{YONG}. 
A.~Fink, K.~M\'ez\'aros, and A.~St.~Dizier confirmed the full conjecture in \cite{FINK}.  As noted by the authors of \cite{huh2019logarithmic} the combination of \autoref{prop_poly} (they use the equivalent \cite[Corollary 10.2]{LORENTZ}) and \cite[Theorem 15.40]{MILLER_STURMFELS} (which is also included in \cite[Theorem 6]{huh2019logarithmic}) is enough to give an alternative proof to  \cite[Conjecture 5.5]{YONG}.

\begin{theorem}\label{thm_Schubert}
For any permutation $\pi$, the Schubert polynomial $\fS_\pi$ has SNP and its Newton polytope is a polymatroid polytope.
\end{theorem}

The Newton polytope of a polynomial $f$ is by definition the convex hull of the exponents in the support of $f$, however in by our convention in \autoref{defMsupp} $\msupp$ consists of the complementary exponents. This does not change the conclusion that the resulting polytope is a polymatroid polytope.

\medskip

{\bf Codimensions of projections.} 
We now use \autoref{thmA} to give a combinatorial interpretation for the codimensions of the natural projections of matrix Schubert varieties. First we need some terminology.

A {\it  diagram} $D$ is a subset of a $p\times p$ grid whose boxes are indexed by  the set $[p]\times [p]$.  
The authors of \cite{YONG} define a function $\theta_D:2^{[p]}\mapsto \ZZ$ as follows: for a subset $\fJ\subseteq [p]$ and   $c\in [p]$ ,  we construct a word $W_D^c(\fJ)$ by reading the column $c$ of $[p]\times [p]$ from top to bottom  and recording
\begin{itemize}
\item $($ if $(r,c)\notin D$ and $r\in \fJ$,
\item $)$ if $(r,c)\in D$ and $r\notin \fJ$,
\item $\star$ if $(r,c)\in D$ and $r\in \fJ$;
\end{itemize}
 let $\theta_D^c(\fJ)=\#\textrm{ paired ``()" in }W_D^c(\fJ)+\#\star\textrm{ in } W_D^c(\fJ)$, and finally $\theta_D(\fJ)=\sum_{i=1}^p \theta_D^i(\fJ)$.
\begin{example}\label{ex_diagram}
For examplae, let $D$ be the diagram depicted in  \autoref{fig:diagram} and $\fJ=\{2,3\}$, then $\theta_D(\fJ)=3$. 

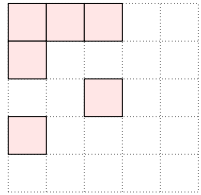
\begin{figure}[ht]
  \scalebox{.5}{
	\begin{tikzpicture}

	\draw[dotted] (0,0)--(5,0)--(5,5)--(0,5)--(0,0);
	\draw[dotted] (1,0)--(1,5);
	\draw[dotted] (2,0)--(2,5);
	\draw[dotted] (3,0)--(3,5);
	\draw[dotted] (4,0)--(4,5);
	\draw[dotted] (0,1)--(5,1);
	\draw[dotted] (0,2)--(5,2);
	\draw[dotted] (0,3)--(5,3);
	\draw[dotted] (0,4)--(5,4);
	\draw[fill=red!10] (0,5)--(1,5)--(1,4)--(0,4)--cycle;
	\draw[fill=red!10] (1,5)--(2,5)--(2,4)--(1,4)--cycle;
	\draw[fill=red!10](2,5)--(3,5)--(3,4)--(2,4)--cycle;
	\draw[fill=red!10](0,4)--(1,4)--(1,3)--(0,3)--cycle;
	\draw[fill=red!10](0,2)--(0,1)--(1,1)--(1,2)--cycle;
	\draw[fill=red!10] (2,3)--(2,2)--(3,2)--(3,3)--cycle;\
		\end{tikzpicture} }
	\caption{Example of a diagram in $[5]\times[5]$}
	\label{fig:diagram} 
\end{figure}
\end{example}

For any $\pi\in\fS_p$ we can define its {\it Rothe diagram} as  $$D_\pi:=\{(i,j)\mid 1\leq i,j\leq n, \pi(i)>j\text{ and }\pi^{-1}(j)>i\}\subset[p]\times[p].$$ For example when $\pi=42531$ then $D_\pi$ is the diagram of \autoref{fig:diagram}.

\begin{theorem} \label{thm_projections_schubert}
 Let $\pi\in\fS_p$, then 
for any $\fJ\subseteq [p]$ the projection $\Pi_\fJ\left(\overline{X_\pi}\right)$ onto the rows indexed by $\fJ$ has codimension $\theta_{D_\pi}([p])-\theta_{D_\pi}(\fJ')$, where $\fJ'=[p]\backslash\fJ$ is the complement of $\fJ$.
\end{theorem}

\begin{proof}
In \cite[Theorem 10]{FINK} the authors givea proof of  \cite[Conjecture 5.13]{YONG}, which in our setup (recall the indexing in \autoref{defMsupp}) states that $\msupp(\overline{X_\pi})$ is equal to
\[
\left\{\bn \in \NN^p \;\mid\; \sum_{j\in\fJ} \left((p-1)-n_j \right)\leq \theta_{D_\pi}(\fJ), \;\forall \fJ\subsetneq [p], \;\sum_{j\in [p]}\left((p-1)-n_j \right)=\theta_{D_\pi}([p])\right\}.
\]
The first inequalities can be rewritten as $(p-1)|\fJ|-\theta_{D_\pi}(\fJ)\leq \sum_{j\in\fJ}n_j$, and combining them with $(p-1)p-\theta_{D_\pi}([p])=\sum_{j\in[p]}n_i$ we obtain
\[
\sum_{j\in \fJ'} n_j \leq (p-1)|\fJ'|-\bigg( \theta_{D_\pi}([p])-\theta_{D_\pi}(\fJ)\bigg).
\]
By \autoref{thmA} we must have that $\codim\left(\Pi_{\fJ}\left(\overline{X_\pi}\right)\right)= \theta_{D_\pi}([p])-\theta_{D_\pi}(\fJ')$ for every $\fJ\subseteq [p]$,  as we wanted to show.
\end{proof}

\begin{remark}
Notice that $\theta_{D_\pi}([p])$ counts the total number of boxes in $D_\pi$, which is equal to the length of $\pi$ (see \cite[Definition 15.13]{MILLER_STURMFELS}). So the case $\fJ=[p]$ of \autoref{thm_projections_schubert} above is equivalent to the well-known fact that the codimesion of  a matrix Schubert variety is equal to the length of the permutation (see \cite[Theorem 15.31]{MILLER_STURMFELS}).
\end{remark}

\subsection{Flag varieties} \label{sec_flag}
We now focus on a multiprojective embedding of flag varieties. 
We first review some terminology. 
For more information the reader is referred to \cite{FULTON_TABLEAUX} or \cite{BRION}.

In this subsection we work over an algebraically closed field $\kk$. Consider the \emph{complete flag variety} $Fl(V)$ of a $\kk$-vector space $V$ of dimension $p+1$. This variety parametrizes complete flags, i.e., sequences $V_\bullet := (V_0,\cdots,V_{p+1})$ such that $\{0\}=V_0\subset V_1 \subset V_2 \subset \cdots \subset V_{p} \subset V_{p+1}=V,$
and each $V_i$ is a linear subspace of $V$ of dimension $i$. 
One can embed this variety in a product of Grassmannians $Fl(V) \hookrightarrow \textrm{Gr}(1,V)\times \textrm{Gr}(2,V) \times \cdots\times \textrm{Gr}(p,V)$ as the subvariety cut out by incidence relations. 

Furthermore, each Grassmannian can be embedded in a projective space via the Pl\"ucker embedding $\iota_i:\textrm{Gr}(i,V) \rightarrow\PP_\kk^{m_i}$ for $1\leq i\leq p$. 
By considering the product of these maps, we obtain a multiprojective embedding of $\iota: Fl(V)\hookrightarrow \PP_\kk^{m_1}\times_\kk\cdots \times_\kk\PP_\kk^{m_p}$. 
For convenience we also call $\iota$ the {\it Pl\"ucker embedding}.
The proposition below computes the corresponding multidegree support.

\begin{proposition}\label{prop_sottile}
Let $V$ be a $\kk$-vector space of dimension $p+1$ and let $X$ be the image of the Pl\"ucker embedding $\iota: Fl(V)\hookrightarrow \PP = \PP_\kk^{m_1}\times_\kk\cdots\times_\kk\PP_\kk^{m_{p}}$, then
\begin{equation}\label{eq:flag_simple}
\msupp_\PP(X)=\left\{\bn\in\NN^{p} \;\mid\; 1\leq n_k\leq \sum_{j=1}^{k}(p-j)-\sum_{i=1}^{k-1} n_i, \;\; \forall k\in [p],\;\; \sum_{j=1}^p n_j = \binom{p+1}{2} \right\};
\end{equation}

\end{proposition}
\begin{proof}
 We need to compute the dimension of $\Pi_\fJ(X))$ for each $\fJ=\{j_1,\ldots,j_k\} \subseteq [p]$. 
The key observation is that $\Pi_\fJ(X)$ is isomorphic to the \emph{partial flag variety} $Fl_{\fJ}(V)$: it parametrizes flags $W_\bullet := \{0\}=V_0\subset V_{2} \subset V_{2} \subset \cdots \subset V_{k} \subset V_{k+1}=V,$ where $\dim V_{k}=j_k$. Hence
\[
\dim\left(\Pi_\fJ(X)\right) = \dim\left(Fl_{\fJ}(V)\right) = \sum_{1\leq i<j\leq k+1} d_id_j = \mathcal{S}(\fJ);
\]
here, for each $\fJ =\{j_1,\ldots,j_k\} \subseteq [p]$, we set $\mathcal{S}(\fJ) := \sum_{1\leq i<j\leq k+1} d_id_j$, where $d_i:=j_i-j_{i-1}$ and by convention $j_0:=0,\, j_{k+1}:=p+1$.
For a proof of the second equality see \cite[\S 1.2]{BRION}. From \autoref{thmA} it follows that
\begin{equation}\label{eq:flag}
\msupp_\PP(X)=\left\{\bn\in\NN^{p} \;\mid\; \sum_{j\in\fJ}n_j\leq \mathcal{S}(\fJ), \;\; \forall\fJ\subseteq[p],\;\; \sum_{j=1}^p n_j = \binom{p+1}{2} \right\}.
\end{equation}
It can be checked that the description  in \autoref{eq:flag} coincides with the one in  \autoref{eq:flag_simple}.
\end{proof}

The pullbacks of the classes $H_i$ from $\PP_\kk^{m_1}\times_\kk\cdots\times_\kk\PP_\kk^{m_p}$ to $Fl(V)$ are called the Schubert divisors, so \autoref{prop_sottile} amounts to a criterion for which powers of these classes intersect. These intersections are called Grassmannian Schubert problems in \cite{PURBHOO_SOTTILE}. In \cite[Theorem 1.2]{PURBHOO_SOTTILE} K.~Purbhoo and F.~Sottile give a stronger statement by providing an explicit combinatorial formula using \emph{filtered tableau} to compute the exact intersection numbers.

\subsection{A multiprojective embedding of $\overline{M}_{0,p+3}$}\label{sec_M0}

The moduli space $\overline{M}_{0,p+3}$ parametrizes rational stable curves with $p+3$ marked points. 
Here we apply our methods to an embedding considered in \cite{MONKS}. 
The starting point is the closed embedding $\Psi_p: \overline{M}_{0,p+3}\longrightarrow \overline{M}_{0,p+2}\times_\kk \PP_\kk^p$ constructed by S.~Keel and J.~Tevelev in \cite[Corollary 2.7]{KEEL_TEVELEV}. 
By iterating this construction we obtain an embedding $\overline{M}_{0,p+3}\hookrightarrow \PP_\kk^1\times_\kk\PP_\kk^2\times_\kk\cdots\times_\kk\PP_\kk^p$ (see \cite[Corollary 3.2]{MONKS}). 
In \cite{MONKS}, R.~Cavalieri, M.~Gillespie, and L.~Monin computed the corresponding multidegree which turns out to be related to parking functions. 
As an easy consequence of our \autoref{thmA}, we can compute its support.

\begin{proposition}\label{prop_parking}
	Let $X$ be the image of  $\overline{M}_{0,p+3}\hookrightarrow \PP =  \PP_\kk^1\times_\kk\PP_\kk^2\times_\kk\cdots\times_\kk\PP_\kk^p$, then 
	\begin{equation}\label{eq:M0n}
	\msupp_\PP(X)=\left\{\bn\in\NN^p\;\mid\; \sum_{i=1}^k n_i\leq k,\; \forall 1\leq k\leq p-1,\; \sum_{i=1}^p n_i=p\right\}.
	\end{equation}

\end{proposition}

\begin{proof}
	First, as explained in \cite[\S 3]{MONKS} we have $\dim\left(\Pi_{[p]}(X)\right)=\dim\left( \Pi_{\{p\}}(X)\right)=p$. 
	Also, by construction $\Pi_{[p-1]}(X)\cong\overline{M}_{0,p+2}$, and thus $\dim \left(\Pi_{[i-1]}(X)\right)=i-1$ for every $2\le i\le p$.  
	So, by induction one gets $\dim\left(\Pi_{\{i\}}(X)\right)=i$ for all $1 \le i \le p$.
	
	To use \autoref{thmA}, we must compute $\dim \left(\Pi_{\fJ}(X)\right)$ for all $\fJ \subseteq [p]$. 
	Let $m:=\max\{j \mid j \in\fJ\}$, then as explained above we have 
	$\dim\left(\Pi_{[m]}(X)\right)=\dim\left(\Pi_{\{m\}}(X)\right)=m$ 
	and so we must have $\dim\left(\Pi_{\fJ}(X)\right)=m$. By \autoref{thmA} we obtain that $\msupp_\PP(X)$ is equal to 
	$$\{\bn\in\NN^p \,\mid\, \sum_{i\in\fJ} n_i\leq \max\{j \mid j \in\fJ\},\; \forall \fJ\subseteq[p],\; \sum_{i=1}^p n_i=p\},
	$$ but it is straightforward to check that the inequalities in \autoref{eq:M0n} are enough to describe the same set.
\end{proof}

The cardinality of $\msupp_\PP(X)$ is equal to the Catalan number $C_n$ (see \cite[Exercise 86]{CATALAN}). For a comprehensible survey on Catalan numbers see \cite{CATALAN}.

\subsection{Mixed Volumes}\label{sec_mixed}
In this subsection we assume $\kk$ is an algebraically closed field. We begin by  reviewing the definition of mixed volumes of convex bodies, as a general reference see \cite[Chapter IV]{EWALD}.
Let $\textbf{K}=(K_1,\ldots,K_p)$ be a $p$-tuple of convex bodies in $\mathbb{R}^d$.  The volume polynomial $v(\textbf{K})\in \ZZ[w_1,\ldots,w_p]$ is defined as
\[
v(K_1,\ldots,K_p) :=  \text{Vol}_d(w_1K_1+\cdots+w_pK_p).
\]
This is a homogeneous polynomial of degree $d$. If the coefficients of $v(\textbf{K})$  are written as $\binom{d}{\bn}V(\textbf{K};\bn)w^\bn$, then the numbers $V(\textbf{K};\bn)$ are called the {\it mixed volumes of} $\textbf{K}$. A natural question to ask is: when are mixed volumes positive? The relation between mixed volumes and toric varieties (see  \autoref{eq_mixed-intersection} below) together with \autoref{thmA} allows us to give another proof of a classical theorem  formulated on the non-vanishing of mixed volumes \cite[Theorem 5.1.8]{SCHNEIDER}. 

\begin{theorem}
	\label{thm_mixed}
	Let $\textnormal{\textbf{K}}=(K_1,\ldots,K_p)$ be a $p$-tuple of convex bodies in $\mathbb{R}^d$. 
	Then, $V(\textnormal{\textbf{K}};\bn)>0$ if and only if $\sum_{i=1}^p n_i=d$ and $\sum_{i\in \fJ} n_i\leq \dim\left(\sum_{i\in \fJ} K_i\right)$ for every subset $\fJ \subseteq [p]$.
\end{theorem}

We first indicate how to reduce to the case of polytopes. The basic idea is that convex bodies can be approximated by polytopes in the Hausdorff metric \cite[Section 1.8]{SCHNEIDER}. However, the condition for positivity as stated in \autoref{thm_mixed} is a priori not stable under limits.
To fix this we invoke an equivalent condition more suitable for the limiting argument.

\begin{lemma}\label{lem_reduction}
	It suffices to show  \autoref{thm_mixed} 
	 for polytopes.
\end{lemma}
\begin{proof}
	This follows from two facts.
	The first is that mixed volumes $V(\textnormal{\textbf{K}};\bn)$ are continuous \cite[Theorem 5.1.7]{SCHNEIDER} and monotonous \cite[Equation 5.25]{SCHNEIDER} on each entry. 
	The second fact is that for a given sequence $\textnormal{\textbf{K}}=(K_1,\ldots,K_p)\subset\left(\mathbb{R}^d\right)^p$ of convex bodies, by \cite[Lemma 5.1.9]{SCHNEIDER} the following conditions are equivalent:
	\begin{enumerate}
		\item $\sum_{i=1}^p n_i=d$ and $\sum_{i\in \fJ} n_i\leq \dim\left(\sum_{i\in \fJ} K_i\right)$ for every subset $\fJ\subseteq[p]$.
		\item \label{cond_segments} There exist line segments $S_{i,1},S_{i,2},\ldots,S_{i,n_i}\subseteq K_i$, for every $i$, such that $\{S_{i,j}\}_{1\ls i\ls p,1\ls j\ls n_i}$ has segments in $d$ linearly independent directions.
	\end{enumerate}
	We now assume the statement of \autoref{thm_mixed} is true when each $K_i$ is a polytope and show that it follows in the case where each $K_i$ is an arbitrary convex body. 
	
	If $V(\textnormal{\textbf{K}};\bn)>0$ then by continuity we can find polytopes $\textnormal{\textbf{P}}=(P_1,\cdots,P_p)$ with $V(\textnormal{\textbf{P}};\bn)>0$ and $P_i \subseteq K_i$ for each $i\in[p]$. By assumption, the sequence $\textnormal{\textbf{P}}$ satisfies condition \hyperref[cond_segments]{(2)} above and hence so does the sequence $\textnormal{\textbf{K}}$. 
	
	Conversely, suppose condition \hyperref[cond_segments]{(2)} holds for $\textnormal{\textbf{K}}$, then by continuity we can find polytopes $\textnormal{\textbf{P}}=(P_1,\ldots,P_p)$ with $P_i\subseteq K_i$ arbitrarily close so that \hyperref[cond_segments]{(2)}  holds for $\textnormal{\textbf{P}}$ too. 
	Due to the assumption, we have $V(\textnormal{\textbf{P}};\bn)>0$ and thus $V(\textnormal{\textbf{K}};\bn)\geq V(\textnormal{\textbf{P}};\bn)>0$ by monotonicity.
\end{proof}

To finish the proof of \autoref{thm_mixed} we need some preliminary results about toric varieties and lattice polytopes.
As an initial step we recall some facts about basepoint free divisors; a general reference is \cite[Section 6]{COX_TORIC}. Let $\Sigma$ be  a fan and  let $P$ be a lattice polytope  whose normal fan coarsens $\Sigma$.  
Then,  $P$ induces a basepoint free divisor $D_P$ in the toric variety $Y_{\Sigma}$ \cite[Proposition 6.2.5]{COX_TORIC}. Here, being basepoint free means that the complete linear series $|D_P|$ induces a morphism $\phi_{P}:Y_{\Sigma}\rightarrow \PP_{\kk}^{m_i}$ for some $m_i\in\NN$ such that $\phi^*(H)=D_P\in A^*(Y_{\Sigma})$, where $H$ is the class of a hyperplane in the projective space $\PP_{\kk}^{m_i}$.

\begin{lemma}\label{lem_toric-aux-1}	
	Let $K_1,\ldots,K_p$ be lattice polytopes and let $K:=K_1+\cdots+K_p$ be their Minkowski sum . Let  $Y$  be the toric variety associated to $\Sigma$, the normal fan of $K$, then  for each $\fJ\subseteq[p]$ we have a map $\phi_{\fJ}: Y\rightarrow \prod_{j\in\fJ} \PP_{\kk}^{m_j}$ such that $\dim\left(\phi_{\fJ}(Y)\right)=\dim\left(\sum_{i\in\fJ}K_i\right)$.
\end{lemma}

\begin{proof}	
	The fan $\Sigma$ is the common refinement of the normal fans of $K_1,\ldots,K_p$ \cite[Proposition 7.12]{ZIEGLER}, so each $K_i$ induces a basepoint free divisor $D_i$ on $Y=Y_\Sigma$ and thus also a map $\phi_i: Y\rightarrow \PP_{\kk}^{m_i}$. By the universal property of fiber products these maps induce a canonical map $\phi_{\fJ}: Y\rightarrow \prod_{j\in\fJ} \PP_{\kk}^{m_j}$ for each $\fJ\subseteq[p]$. It remains to compute the dimensions of the images of these maps.
	
	By composing with the Segre embedding $\varphi_{\fJ}:\prod_{j\in\fJ} \PP_{\kk}^{m_j}\rightarrow \PP_{\kk}^m$ we obtain a map $\varphi_{\fJ}\circ\phi_{\fJ}: Y\rightarrow \PP_{\kk}^m$. Let $H$ be the class of a hyperplane in $A^*(\PP_{\kk}^m)$, we have that $\varphi_{\fJ}^*(H)=\sum_{i\in\fJ} H_i\in A^*\left(\prod_{j\in\fJ} \PP_{\kk}^{m_j}\right)$ where the each $H_i$ is the pullback of a hyperplane in the $i$-th factor (see, e.g., \cite[Exercise 5.11]{HARTSHORNE}). 
	Then $(\varphi_{\fJ}\circ\phi_{\fJ})^*(H)=\sum_{i\in\fJ} D_i\in A^*(Y)$. This means that the morphism $\varphi_{\fJ}\circ\phi_{\fJ}$ corresponds to the complete linear series $|D_{K_{\fJ}}|$ where $K_{\fJ}=\sum_{j\in\fJ} K_i$, hence the image has dimension $\dim(K_{\fJ})=\dim(\sum_{j\in\fJ} K_i)$ \cite[Theorem 6.1.22]{COX_TORIC}.
\end{proof}

\begin{lemma} \label{lem_toric-aux-2}	
	In the setup of \autoref{lem_toric-aux-1}, if $\fJ=[p]$ then after scaling each polytope if necessary, $\phi = \phi_{\fJ}$ is an embedding.
\end{lemma}
\begin{proof}
	By construction the normal fan of $K_{[p]}=\sum_{i=1}^p K_i$ is $\Sigma$, so the corresponding divisor is ample \cite[Theorem 6.1.14]{COX_TORIC}. By replacing the list of polytopes by large enough scalings we obtain a very ample divisor, hence an embedding.
\end{proof}

\begin{proof}[Proof of \autoref{thm_mixed}]
	By using \autoref{lem_reduction}, we can assume that each $K_i$ is a polytope. 
	Additionally, we can reduce to the case where each $K_i$ is a lattice polytope since any polytope can be approximated by lattice polytopes (see \cite[Page 120]{FULTON}). 
	Let $K=K_1+\cdots+K_p$ and let $Y$ be the  toric projective variety associated to the normal fan of $K$. Each lattice polytope $K_i$ induces a basepoint free divisor $D_i$ on $Y$. As explained in \cite[Eq. (2), Page 116]{FULTON}, the fundamental connection between mixed volumes and intersection products is  given by the following equation
	\begin{equation}\label{eq_mixed-intersection}
		V(\textbf{K};\bn)=\left(D_1^{n_1}\cdots D_p^{n_p}\right)/d!,
	\end{equation}	
	where the numerator is the intersection product of the divisors in $Y$. Notice that positivity of mixed volumes  is unchanged by scaling so whenever needed we can scale each polytope.

	By \autoref{lem_toric-aux-2} we have an embedding $\phi: Y\rightarrow \prod_{i=1}^p \PP_{\kk}^{m_i}$ such that the pullback of each $H_i\in A^*\left(\prod_{i=1}^p \PP_{\kk}^{m_i}\right)$ is $D_i$. 
	By using the projection formula \cite[Proposition 2.5(c)]{FULTON_INTERSECTION_THEORY}, we can consider the product $\left(H_1^{n_1}\cdots H_p^{n_p}\right)/d!$ instead of the one in \autoref{eq_mixed-intersection}.
	From the fact that $Y$ is irreducible we are now in the \autoref{setup_initial}. So, the result follows by \autoref{rem_chow_ring},  \autoref{thmA}, and \autoref{lem_toric-aux-1}.
\end{proof}



\section*{Acknowledgments}
We thank the reviewer for his/her suggestions for the improvement of this work.
We would like to thank  Chris Eur, Maria Gillespie, June Huh, David Speyer, and Mauricio Velasco for
useful conversations. We are also grateful with Frank Sottile for useful comments on an earlier version (in particular for the simplification of the statement of \autoref{prop_sottile}). Special thanks to  Brian Osserman for many insightful conversations and encouragements. 
The computer algebra system \texttt{Macaulay2} \cite{MACAULAY2} was of great help to compute several examples in the preparation of this paper.

\bibliographystyle{elsarticle-num} 
\bibliography{references}

\end{document}